\documentclass[11pt]{article}       
\usepackage{geometry}               
\usepackage{tikz}
\usetikzlibrary{shapes.geometric, arrows}

\usepackage{graphicx}
\usepackage{caption}
\usepackage{subcaption}
\usepackage{comment}
\usepackage{amsmath} 
\usepackage{amssymb}  
\usepackage{amsthm}  
\usepackage{bm}
\usepackage{lipsum}
\usepackage{color}
\usepackage{multirow}
\usepackage{array}
\usepackage{cite}
\usepackage{hyperref}

\usepackage{algorithm}
\usepackage{algpseudocode}

\newtheorem{lem}{Lemma}
\newtheorem{thm}{Theorem}

\numberwithin{equation}{section}

\title{Physics-Informed Neural Networks for Optimal Vaccination Plan in SIR Epidemic Models}

\author{
Minseok Kim\thanks{Department of Applied Artificial Intelligence, 232 Gonneung-ro, Nowon-gu, Seoul, 01811, Republic of Korea.}
\and
Yeongjong Kim \thanks{Department of Mathematics, POSTECH, 77 Cheongam-ro, Nam-gu, Pohang, Gyeongbuk, 06974, Republic of Korea} 
\and
Yeoneung Kim\footnotemark[2]}

\date{}
\providecommand{\keywords}[1]{\textbf{Key words.} #1}

\begin{document}
\maketitle

\pagestyle{myheadings}
\thispagestyle{plain}

\begin{abstract}
This work focuses on understanding the minimum eradication time for the controlled Susceptible-Infectious-Recovered (SIR) model in the time-homogeneous setting, where the infection and recovery rates are constant. The eradication time is defined as the earliest time the infectious population drops below a given threshold and remains below it. For time-homogeneous models, the eradication time is well-defined due to the predictable dynamics of the infectious population, and optimal control strategies can be systematically studied. We utilize Physics-Informed Neural Networks (PINNs) to solve the partial differential equation (PDE) governing the eradication time and derive the corresponding optimal vaccination control. The PINN framework enables a mesh-free solution to the PDE by embedding the dynamics directly into the loss function of a deep neural network. 
We use a variable scaling method to ensure stable training of PINN and mathematically analyze that this method is effective in our setting. This approach provides an efficient computational alternative to traditional numerical methods, allowing for an approximation of the eradication time and the optimal control strategy. Through numerical experiments, we validate the effectiveness of the proposed method in computing the minimum eradication time and achieving optimal control. This work offers a novel application of PINNs to epidemic modeling, bridging mathematical theory and computational practice for time-homogeneous SIR models.
\end{abstract}


\section{Introduction}
The study of vaccination strategies and eradication times in Susceptible-Infectious-Recovered (SIR) models has a long history, beginning with the seminal work of Kermack and McKendrick \cite{kermack1927SIR}, and its variants have received a great deal of attention during and after the outbreak of COVID-19. The controlled SIR model is a cornerstone in mathematical epidemiology and is frequently employed to study the dynamics of disease transmission and control strategies. Various optimization problems based on the SIR model, where the vaccination strategy is treated as a control and the eradication time as a cost function to be minimized, have been extensively studied within the framework of optimal control theory~\cite{barro2018optimal,pierre-alexandre2007optimal,bolzoni2017time,ev2016optimalcontrol}. 

The controlled SIR model is given by
\[
\begin{cases}
    \dot S = - \beta(t) S I - r(t) S,\\
    \dot I = \beta(t) S I - \gamma(t) I,
\end{cases}
\]
for \(t > 0\), with initial conditions \(S(0) = x\) and \(I(0) = y \geq \mu\) for some \(\mu > 0\). Here, \(\beta(t)\) and \(\gamma(t)\) denote the infection and recovery rates, respectively, while \(r(t)\) represents a vaccination control that takes values in \([0, 1]\). 

Recently, the authors of~\cite{bolzoni2017time} introduced the notion of minimum eradication time, defined as the first time \(I\) falls below a given threshold \(\mu > 0\). For mathematical treatments, the eradication time in controlled SIR models with constant infection and recovery rates was first studied as a viscosity solution to a static first-order Hamilton--Jacobi--Bellman (HJB) equation in \cite{hynd2021eradication}. Additionally, two critical times in SIR dynamics are studied in~\cite{hynd2022critical}: the point at which the infected population begins to decrease and the first time that this population falls below a given threshold. Both studies~\cite{hynd2021eradication,hynd2022critical} focused on SIR models with constant \(\beta\) and \(\gamma\). To identify optimal controls, the Pontryagin Maximum Principle~\cite{pontryagin2018mathematical} was applied, confirming that a bang-bang control (i.e., taking values of 0 or 1) is optimal. The authors of~\cite{jang2023minimum} extended the notion of eradication time to cases involving time-inhomogeneous dynamics. However, numerical treatments for solving Hamilton-Jacobi equations and determining optimal vaccination controls remain relatively unexplored.

Based on the finding that the minimum eradication time satisfies the HJB equation that has a viscosity solution~\cite{hynd2021eradication,tran2021hamilton}, we propose a novel approach leveraging Physics-Informed Neural Networks (PINNs) to approximate the minimum eradication time and synthesize an optimal control for the time-independent controlled SIR model in a mesh-free manner. In the context of optimal control theory, PINNs offer computationally efficient and accurate solutions while traditional numerical approaches to such problems often demonstrate computational challenges. Our study leverages PINNs to bridge the gap between theoretical insights into eradication time and practical control strategies.

Additionally, we provide theoretical evidence for our approach. Despite the numerous successful applications of PINNs, a comprehensive understanding of the conditions required for stable and effective learning remains incomplete. To address this, significant efforts have been made to identify theoretical conditions for successful learning and to develop empirical methods that enhance the stability and reliability of PINNs during training. One simple yet promising method is the variable scaling technique proposed in~\cite{VS-PINN}, inspired by the observation that neural networks often struggle to approximate stiff functions effectively. The authors of~\cite{VS-PINN} provide a mathematical analysis of the variable scaling method in the context of a one-dimensional Poisson equation, leveraging the neural tangent kernel (NTK) theory~\cite{jacot2018neural}. Building on this work, we extend their analysis to the HJB equation, a more complex problem defined over a higher-dimensional domain.

\subsection{Related works}
PINNs~\cite{raissi2019physics} have gained significant attention as a powerful and flexible framework for solving differential equations, and have been widely adopted in various fields, including epidemic modeling~\cite{kharazmi2021identifiability,yazdani2020systems}, fluid mechanics~\cite{cai2021physics,raissi2020hidden,raissi2019physics,jin2021nsfnets}, finance~\cite{wang2023deep,bai2022application}, and biomedical engineering~\cite{kissas2020machine,sahli2020physics}, where understanding the underlying physical models is crucial.

To find a solution to differential equations in the framework of PINNs, we train a neural network to minimize a loss function comprising initial and boundary conditions as well as residual terms derived from the governing equations. However, the training results are highly sensitive to the choice of boundary condition settings, requiring the introduction of a penalty coefficient to balance the boundary loss term. While heuristic adjustments to the penalty coefficient can accelerate convergence, improper values may lead to inaccurate solutions. To address these challenges, adaptive methods have been developed. For instance, the authors of~\cite{wang2021understanding} proposed a learning rate annealing algorithm that adaptively assigns weights to each term in the loss function. PINNs with adaptive weighted loss functions have been introduced for the efficient training of Hamilton–Jacobi (HJ) equations~\cite{liu2022physics}. To further enhance the stability of PINNs, the authors of~\cite{wang2022and} proposed an adaptive training strategy that ensures stable convergence through the lens of NTK theory~\cite{jacot2018neural}. Recently, the failure of PINNs in stiff ODE systems was observed~\cite{ji2021stiff}, and stiff-PINN was proposed for improvement. Subsequently, various methods have been introduced, such as self-adaptive PINNs~\cite{mcclenny2023self} and variable-scaling PINNs~\cite{VS-PINN}. Among these methods, we employ the variable scaling technique~\cite{VS-PINN}, as it is simple and effective.

While PINNs have been successfully applied to a wide range of differential equation problems, their application to optimal control, particularly in solving Hamilton–Jacobi–Bellman (HJB) equations, remains relatively underexplored. The key challenge lies in ensuring stability and accuracy when approximating value functions and control policies. This has motivated recent studies investigating the interplay between deep learning and optimal control, aiming to develop computationally efficient methods that leverage the advantages of PINNs for solving PDE and optimal control problems.

There is a rich body of literature exploring the interplay between PINNs and optimal control. By leveraging the ability of PINNs to solve PDEs and the scalability of deep neural networks, researchers have developed computationally efficient methods for solving optimal control problems. For instance, a training procedure for obtaining optimal control in PDE-constrained problems was presented in~\cite{mowlavi2023optimal}. Similarly, the authors of~\cite{meng2024physics} utilized a Lyapunov-type PDE for efficient policy iteration in control-affine problems. Slightly later, a deep operator learning framework was introduced to solve high-dimensional optimal control problems~\cite{lee2024hamilton}, building on the policy-iteration scheme developed in~\cite{tang2023policy}. Most recently,~\cite{yin2023optimal} demonstrated the application of deep learning in controlled epidemic models.

Building on these advancements, our work focuses on leveraging PINNs for solving optimal control problems, specifically in the context of SIR models with vaccination strategies. The proposed approach not only provides an effective approximation of the minimum eradication time but also facilitates the synthesis of optimal control policies in a computationally efficient manner.

\subsection{Contributions}
This paper makes the following contributions:
\begin{itemize} 
\item We propose a PINN-based computational framework to approximate the minimum eradication time for the controlled SIR model without domain discretization. 
\item We derive optimal vaccination control strategies from the learned eradication time and demonstrate the effectiveness of our method through numerical experiments. 
\item We provide a mathematical analysis, based on the NTK theory, to illustrate the effectiveness of variable scaling in training PINNs for solving the HJB equation. \end{itemize}

\subsection{Organization of the paper}
The remainder of this paper is organized as follows: Section 2 presents preliminary results on the minimum eradication time in the context of HJB equations. Section 3 reviews variable-scaling PINNs and includes an error analysis specific to our HJB equation. Section 4 details the training procedure, while Section 5 presents the experimental results. Finally, Section 6 concludes the paper by summarizing the key findings, discussing the limitations of the approach, and providing directions for future research.

\section{Hamilton--Jacobi--Bellman equation for the minimum eradication time}

\subsection{Minimum eradication time problem}\label{sec:prelim}
Throughout the paper, we consider a time-homogeneous controlled SIR model where the infection and recovery rates are constant:
\[
\beta(t) \equiv \beta \quad \text{and} \quad \gamma(t) \equiv \gamma.
\]
Given a threshold \(\mu > 0\), initial conditions \(x \geq 0\) and \(y \geq \mu\), and a vaccination control \(r(t) \in \mathcal{U} = \{r: [0, \infty) \rightarrow [0, 1]\}\), we define the eradication time as
\[
u^r(x, y) := \min \{t > 0 : I(t) = \mu\},
\]
where \(S^r\) and \(I^r\) satisfy
\begin{equation}\label{eq:controlled_raw}
    \begin{cases}
        \dot S^r = - \beta S^r I^r - r S^r,\\
        \dot I^r = \beta S^r I^r - \gamma I^r,
    \end{cases}
\end{equation}
with $(S^r(0),I^r(0)) = (x,y)$.

A crucial property of \(u^r\) is that for each \(t \in [0, u^r(x, y)]\),
\begin{equation}\label{eq:babyDPP}
u^r(x, y) = t + u^r(S^r(t), I^r(t)),
\end{equation}
which is known as the dynamic programming principle (DPP). This relationship can be interpreted as follows: at time \(t\), the remaining eradication time from the state \((S^r(t),I^r(t))\) is \(u^r(x, y)\).

Finally, the minimum eradication time is defined as
\begin{equation}\label{eq:value_function}
u(x, y) := \min_{r \in \mathcal{U}} u^r(x, y).
\end{equation}
The mathematical properties of this value function have been extensively studied in~\cite{hynd2021eradication}. For the convenience of readers, we summarize the theoretical results provided in~\cite{hynd2021eradication}.

Thanks to~\eqref{eq:babyDPP}, it is known that \(u\) is the unique viscosity solution to the following HJB equation.

\begin{thm}[Theorem 1.2 of~\cite{hynd2021eradication}]
For \(\mu > 0\), the value function \(u\) defined in~\eqref{eq:value_function} is the unique viscosity solution to
\begin{align}\label{eq:hj}
\beta x y \partial_x u + x (\partial_x u)^+ + (\gamma - \beta x) y \partial_y u = 1 \quad\text{in}\quad (0,\infty) \times (\mu, \infty), 
\end{align}
with boundary conditions
\[
u(0, y) = \frac{1}{\gamma} \ln \bigg(\frac{y}{\mu}\bigg) \quad \text{for} \quad y \geq \mu,
\]
and
\[
u(x, \mu) = 0 \quad \text{for} \quad 0 \leq x \leq \frac{\gamma}{\beta}.
\]
\end{thm}

In the next section, we proceed with identifying the optimal control that minimizes the eradication time.

\subsection{Optimal bang-bang control}
It is known from~\cite{bolzoni2017time} that the optimal control \(r\) takes the form of bang-bang control, defined as
\begin{equation}\label{eq:bangbang}
r_\tau(t) =
\begin{cases}
0, &\quad  t < \tau,\\
1, &\quad t \geq \tau.
\end{cases}
\end{equation}
Here, $r_\tau$ and $\tau$ are referred to as the switching control and switching time, respectively. We now recall known results on the minimum eradication time and the optimal switching control.

\begin{thm}[Theorem 1.4 of~\cite{hynd2021eradication}]\label{thm:dpp}
Let \(\mu > 0\), \(x \geq 0\), and \(y \geq \mu\). Then,
\begin{equation}\label{eq:switch}
u(x, y) = \min_{\tau \geq 0} \{\tau + u^{r_0}(S(\tau), I(\tau))\},
\end{equation}
where \(S(t)\) and \(I(t)\) satisfy the uncontrolled SIR model:
\begin{equation}\label{eq:ode_uncontrolled}
\begin{cases}
\dot S = -\beta S I,\\
\dot I = \beta S I - \gamma I,
\end{cases}
\end{equation}
with $(S(0),I(0)) = (x,y)$. Moreover, any $\tau$ for which the minimum in \eqref{eq:switch} is achieved is the switching time of optimal switching control.
\end{thm}

To solve for the optimal switching time $\tau$, it is necessary to compute \(u^{r_0}\). Recalling the DPP~\eqref{eq:babyDPP}, when \(r \equiv 1\), i.e., \(r = r_0\), we have the following identity:
\begin{equation}\label{eq:r0}
u^{r_0}(x, y) = t + u^{r_0}(S^{r_0}(t), I^{r_0}(t)),
\end{equation}
where \(S^{r_0}\) and \(I^{r_0}\) satisfy
\begin{equation}\label{eq:controlled}
\begin{cases}
\dot S^{r_0} &= - \beta S^{r_0} I^{r_0} - S^{r_0},\\
\dot I^{r_0} &= \beta S^{r_0} I^{r_0} - \gamma I^{r_0},
\end{cases}
\end{equation}
with $(S^{r_0}(0),I^{r_0}(0)) = (x,y)$. Taking the time derivative of both sides of~\eqref{eq:r0}, we deduce that
\begin{equation*}
\begin{split}
0 &= 1 + \frac{d}{dt} u^{r_0}(S^{r_0}(t), I^{r_0}(t)) \\
&= 1 + \dot S^{r_0}(t) \partial_x u^{r_0} + \dot I^{r_0}(t) \partial_y u^{r_0} \\
&= 1 + (-\beta S^{r_0}(t) I^{r_0}(t) - S^{r_0}(t)) \partial_x u^{r_0} + (\beta S^{r_0}(t) I^{r_0}(t) - \gamma I^{r_0}(t)) \partial_y u^{r_0}.
\end{split}
\end{equation*}
Setting \(t = 0\) and using the dynamics~\eqref{eq:controlled}, we obtain:
\begin{equation*}\label{eq:u0}
\beta x y \partial_x u^{r_0} + x \partial_x u^{r_0} + (\gamma - \beta x) y \partial_y u^{r_0} = 1 \quad \text{in}\quad (0, \infty) \times (\mu, \infty),
\end{equation*}
where \(u^{r_0}\) satisfies the same boundary conditions as \(u\). 

We finish this section with well-established properties of the optimal switching time $\tau^*$ such that $u(x,y)=u^{\tau^*}(x,y)$. Defining 
\[
\mathcal{S}:=\{(x,y): u(x,y) = u^{r_0}(x,y)\},
\]
Corollary 6.2 in \cite{hynd2021eradication} yields that
\begin{equation}\label{eq:u0_prop}
\begin{cases}
\partial_x u^{r_0}(x,y) \geq 0 \quad &\text{for}\quad (x,y) \in \mathcal{S},\\ 
\partial_x u^{r_0}(S(\tau^*),I(\tau^*)) = 0 \quad &\text{for} \quad (x,y) \in \mathcal{S}^C ,
\end{cases}
\end{equation}
where $(S,I)$ satisfies~\eqref{eq:ode_uncontrolled} with $(S(0),I(0))=(x,y)$.

Furthermore, by Corollary 6.4 in~\cite{hynd2021eradication}, we have that 
\begin{equation}\label{eq:u_prop}
\begin{cases}
\partial_x u(x,y) \geq 0 \quad &\text{for}\quad (x,y) \in \mathcal{S},\\
\partial_x u(x,y) \leq 0 \quad &\text{for}\quad (x,y) \in \mathcal{S}^C.
\end{cases}
\end{equation}

Based on the theoretical properties of the minimum eradication time, we propose a training procedure to solve~\eqref{eq:hj} and compute the optimal switching time through a PINN framework, which does not require spatial discretization.

\section{Variable-scaling physics-informed neural networks}
In this section, we explain the variable-scaling physics-informed neural network (VS-PINN), which is a crucial component of our method. For completeness, we begin with an overview of PINNs.

\subsection{Physics-informed neural networks}
PINNs are trained using a loss function that enables the neural network to approximate a solution satisfying both the differential equation and the initial or boundary conditions. Specifically, we focus on solving a partial differential equation (PDE) with a boundary condition, as the problem we address falls into this category. Suppose we have a bounded open domain \(\Omega \subset \mathbb{R}^d\) and the following equations:
\begin{align*}
    \mathcal{D}[u](\mathbf{x}) = f(\mathbf{x}) &\quad \text{in}\quad \Omega,\\
    u(\mathbf{x}) = g(\mathbf{x})&\quad\text{on}\quad \partial\Omega,
\end{align*}
where \(\mathcal{D}\) is a differential operator. 

We train a neural network \(u(x; \theta)\) using the loss function
\begin{equation}\label{eq:loss}
\mathcal{L} = \lambda_r \mathcal{L}_r + \lambda_b \mathcal{L}_b,
\end{equation}
where the residual loss \(\mathcal{L}_r\) and the boundary loss \(\mathcal{L}_b\) are defined 
as 
\[
\mathcal{L}_r = \frac{1}{N_r} \sum_{i=1}^{N_r} |\mathcal{D}[u](\mathbf{x}_r^i) - f(\mathbf{x}_r^i)|^2, \quad
\mathcal{L}_b = \frac{1}{N_b} \sum_{j=1}^{N_b} |u(\mathbf{x}_b^j) - g(\mathbf{x}_b^j)|^2.\]
Here, the residual data points \(\mathbf{x}_r^i \in \Omega\) and the boundary data points \(\mathbf{x}_b^j \in \partial\Omega\) are typically sampled randomly from uniform distributions. The weights \(\lambda_r\), \(\lambda_b\), and the number of data points \(N_r\), \(N_b\) are tunable parameters.

In our problem, we consider \(\Omega = (0, \infty) \times (\mu, \infty)\) for $\mu>0$ and define the operators \(\mathcal{D}\) and \(\mathcal{D}^0\) as
\begin{equation}\label{problem}
\begin{split}
    \mathcal{D} &= \beta xy \partial_x + x (\partial_x)^+ + (\gamma - \beta x)y \partial_y, \\
    \mathcal{D}^0 &= \beta xy \partial_x + x \partial_x + (\gamma - \beta x)y \partial_y,
\end{split}
\end{equation}
where \((\partial_x)^+u = (\partial_x u)^+\). We solve for \(u\) and \(u^{r_0}\) satisfying
\begin{equation}\label{eq:problem_main}
    \mathcal{D}[u] = 1 \quad \text{and} \quad \mathcal{D}^0[u^{r_0}] = 1 \quad\text{in}\quad \Omega,
\end{equation}
with the boundary conditions
\begin{equation}\label{g}
\begin{cases}
    u(0, y) = u^{r_0}(0, y) = \frac{1}{\gamma}\ln\bigg(\frac{y}{\mu}\bigg) &\quad \text{for } y \geq \mu,\\
    u(x, \mu) = u^{r_0}(x, \mu) = 0 &\quad \text{for } 0 \leq x \leq \frac{\gamma}{\beta}.
\end{cases}
\end{equation}
A schematic diagram of the framework is presented in Figure~\ref{fig:framework_u_pinn}.

\begin{figure}[ht]
    \centering
    \includegraphics[width=\linewidth]{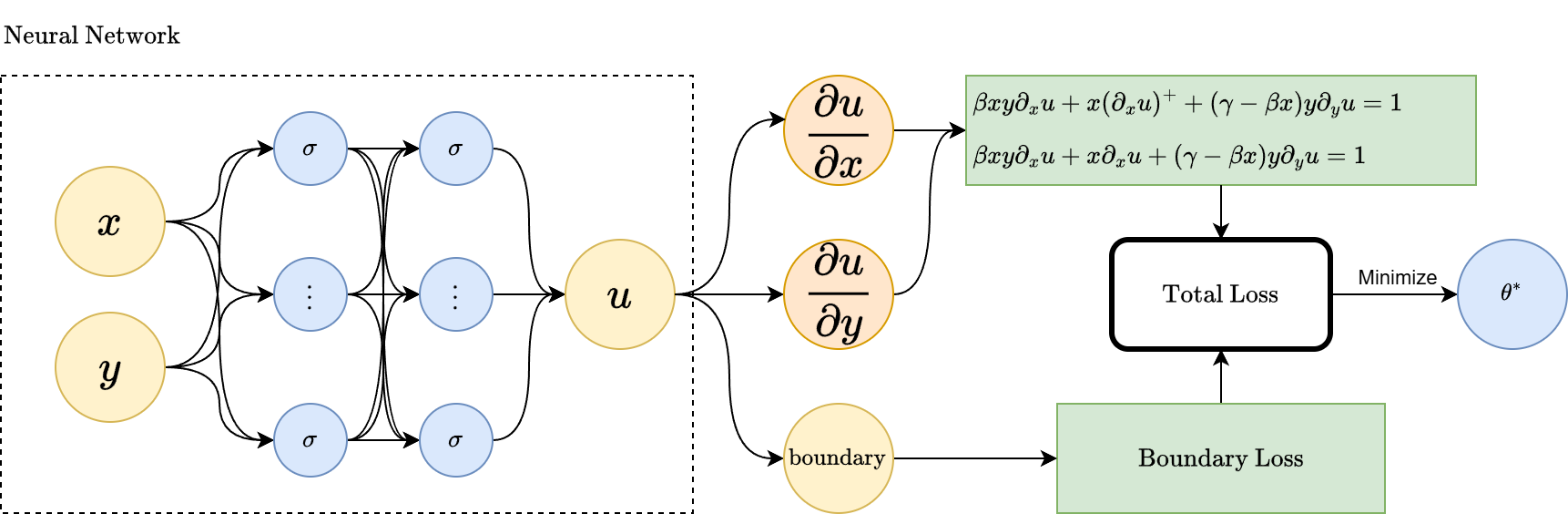}
    \caption{Training \(u\) and \(u^{r_0}\) under the PINN framework.}
    \label{fig:framework_u_pinn}
\end{figure}

\subsection{Variable scaling}
\label{sec:variable_scaling}
In~\cite{VS-PINN}, the authors proposed a simple method that improves the performance of PINNs. The idea is to scale the variables so that the domain of the target function is magnified, making the function less stiff.

Let us begin with a change of variable $\mathbf{x}={\hat{\mathbf{x}}}/{N}$ and setting
\[\hat{u}(\hat{\mathbf{x}}):=u(\hat{\mathbf{x}}/N).\]
The function $\hat{u}$ is a scaled version of the function $u$, with its domain expanded by a factor of $N$, causing its stiffness to decrease by a factor of 
$N$ compared to $u$. Thus, our goal is to train $\hat{u}$ instead of directly training $u$. After training $\hat{u}$, we can simply recover our original target $u$ by substituting
\[u(\mathbf{x})=\hat{u}(N\mathbf{x}).\]

In our problem, we will use different scaling parameters $N_x$, $N_y$ for each component $x, y$ and also apply translations, which means we use
\[\hat{x}=N_x x+b_x,\quad \hat{y}=N_y y+b_y.\]
Thus, after training $\hat{u}$, we recover $u$ by 
\[u(x,y)=\hat{u}(N_x x+b_x, N_y y+b_y).\]
Recall that the differential operator of our problem in ~\eqref{problem} takes the form of $\mathcal{D}=F(x, y, \partial_x, \partial_y)$.
The modified version of our problem has the domain $\hat{\Omega} = (b_x, \infty)\times (N_y\mu+b_y, \infty)$. 

By the chain rule,
\begin{align*}
    &\partial_x u(x,y) = \partial_{\hat{x}}u(x,y) \cdot \partial_x \hat{x} = \partial_{\hat{x}}\hat{u}(\hat{x}, \hat{y}) \cdot N_x,\\
    &\partial_y u(x,y) = \partial_{\hat{y}}u(x,y) \cdot \partial_y \hat{y} = \partial_{\hat{y}}\hat{u}(\hat{x}, \hat{y}) \cdot N_y,
\end{align*}
we define operators $\hat{\mathcal{D}}$ and $\hat{\mathcal{D}^0}$ as
\begin{equation*}
    \hat{\mathcal{D}} = \beta \frac{(\hat{x}-b_x)(\hat{y}-b_y)}{N_y}\partial_{\hat{x}} + (\hat{x}-b_x)(\partial_{\hat{x}})^+ +(\gamma-\beta \frac{\hat{x}-b_x}{N_x})(\hat{y}-b_y)\partial_{\hat{y}},
\end{equation*}
and
\begin{equation*}
    \hat{\mathcal{D}^0} = \beta \frac{(\hat{x}-b_x)(\hat{y}-b_y)}{N_y}\partial_{\hat{x}} + (\hat{x}-b_x)\partial_{\hat{x}} +(\gamma-\beta \frac{\hat{x}-b_x}{N_x})(\hat{y}-b_y)\partial_{\hat{y}}.
\end{equation*}
We then solve for $\hat{u}$ and $\hat{u}^{r_0}$ satisfying
\[\hat{\mathcal{D}}[\hat{u}]=1\quad\text{and}\quad \hat{\mathcal{D}^0}[\hat{u}^{r_0}]=1 \quad\text{in}\quad\hat\Omega\]
with
\begin{equation}\label{g-bar}
\begin{cases}
    \hat{u}(0,\hat{y}) = \hat{u}^{r_0}(0,\hat{y}) = \frac{1}{\gamma}\ln\bigg(\frac{\hat{y}-b_y}{N_y\mu}\bigg)\quad&\text{for}\quad \hat{y}\geq N_y\mu+b_y,\\
    \hat{u}(\hat{x},\mu) = \hat{u}^{r_0}(\hat{x},\mu)=0\quad&\text{for}\quad b_x\leq \hat{x}\leq \frac{N_x\gamma}{\beta}+b_x.
\end{cases}
\end{equation}
As a result, we should use the loss function in \eqref{eq:loss} modified by replacing $\mathcal{D}$ (or $\mathcal{D}^0$), $f, g, u(\mathbf{x};\theta)$ (or $u^{r_0}(\mathbf{x};\theta)$)  with $\hat{\mathcal{D}}$ (or $\hat{\mathcal{D}^0}$), $\hat{f}, \hat{g}, \hat{u}(\hat{\mathbf{x}};\theta)$ (or $\hat{u}^{r_0}(\hat{\mathbf{x}};\theta)$) in order to train $\hat{u}(\hat{\mathbf{x}};\theta)$ (or $\hat{u}^{r_0}(\hat{\mathbf{x}};\theta)$), where $f(x) = \hat{f}(\hat{x})=1$, $g$ and $\hat{g}$ are given in \eqref{g} and \eqref{g-bar} respectively, that is,
\[
\begin{cases}
    g(0, y) = \frac{1}{\gamma}\ln\bigg(\frac{y}{\mu}\bigg) &\quad \text{for } y \geq \mu,\\
    g(x, \mu) = 0 &\quad \text{for } 0 \leq x \leq \frac{\gamma}{\beta},
\end{cases}
\]
and
\[
\begin{cases}
    \hat{g}(0,\hat{y}) = \frac{1}{\gamma}\ln\bigg(\frac{\hat{y}-b_y}{N_y\mu}\bigg)\quad&\text{for}\quad \hat{y}\geq N_y\mu+b_y,\\
    \hat{g}(\hat{x},\mu) = 0\quad&\text{for}\quad b_x\leq \hat{x}\leq \frac{N_x\gamma}{\beta}+b_x.
\end{cases}
\]

In the following section, we analyze the effect of the scaling factors.

\subsection{Theoretical support via neural tangent kernel}
\label{sec:ntk}
We now establish the theoretical foundation for the efficiency of the variable scaling method in our problem through the lens of Neural Tangent Kernel (NTK)~\cite{jacot2018neural}, a widely used framework for analyzing the training dynamics of deep neural networks. In \cite{wang2022and}, the NTK theory was extended to PINNs. Applying this theory to a simple one-dimensional Poisson equation with Dirichlet boundary conditions, the authors of \cite{VS-PINN} demonstrated that the variable scaling method enhances training efficiency in PINNs. Building on these results, we establish similar arguments to the equation (\ref{eq:hj}) and illustrate the effectiveness of the variable scaling method in our setting. To avoid redundancy, we focus on solving for \(u\).

\subsubsection{Neural tangent kernel}
Given that the parameter $\theta$ of a PINN $u(\mathbf{x};\theta)$ is trained through the gradient flow 
\[\frac{d\theta}{dt}=-\nabla_\theta \mathcal{L}\]
with respect to the loss (\ref{eq:loss}) with $\lambda_r=\lambda_b=1/2$, it is proved in \cite{wang2022and} that the evolution of $u$ and $\mathcal{D}[u]$ follows

\begin{equation*}
\begin{bmatrix}
\frac{d u(\mathbf{x}_b; \theta(t))}{d t} \\
\frac{d\mathcal{D}[u](\mathbf{x}_r; \theta(t))}{d t}
\end{bmatrix}
=
-
\begin{bmatrix}
K_{uu}(t) & K_{ur}(t) \\
K_{ru}(t) & K_{rr}(t)
\end{bmatrix}
\begin{bmatrix}
u(\mathbf{x}_b; \theta(t)) - g(\mathbf{x}_b) \\
\mathcal{D}[u](\mathbf{x}_r; \theta(t)) - f(\mathbf{x}_r)
\end{bmatrix},
\end{equation*}
where \(K_{uu}(t) \in \mathbb{R}^{N_b \times N_b}\), \(K_{rr}(t) \in \mathbb{R}^{N_r \times N_r}\), and \(K_{ru}(t) = [K_{ur}(t)]^\top \in \mathbb{R}^{N_r \times N_b}\), whose \((i, j)\)th entries are given by
\begin{align*}
(K_{uu})_{ij}(t) &= \left\langle \frac{d u(\mathbf{x}_b^i; \theta(t))}{d \theta}, \frac{d u(\mathbf{x}_b^j; \theta(t))}{d \theta} \right\rangle, \\
(K_{ru})_{ij}(t) &= \left\langle \frac{d \mathcal{D}[u](\mathbf{x}_r^i; \theta(t))}{d \theta}, \frac{d u(\mathbf{x}_b^j; \theta(t))}{d \theta} \right\rangle, \\
(K_{rr})_{ij}(t) &= \left\langle \frac{d \mathcal{D}[u](\mathbf{x}_r^i; \theta(t))}{d \theta}, \frac{d \mathcal{D}[u](\mathbf{x}_r^j; \theta(t))}{d \theta} \right\rangle.
\end{align*}
We call the matrix
$
K(t) =
\begin{bmatrix}
K_{uu}(t) & K_{ur}(t) \\
K_{ru}(t) & K_{rr}(t)
\end{bmatrix}
$
the NTK of a training dynamics of $u(\mathbf{x};\theta)$ via PINN. 

In \cite{wang2022and}, it is proved that the NTK of PINN at initialization, i.e., the kernel at $t=0$, converges to a deterministic kernel and remains the same in the infinite-width limit when $\theta(0)$ is assumed to follow the standard normal distribution. In \cite{VS-PINN}, it is presented that variable scaling can enhance the performance of PINNs by analyzing how the initial NTK of the one-dimensional Poisson equation evolves with variable scaling. We follow their argument to deduce a similar result for the equation~\eqref{eq:hj}.

By definition, $K(t)$ is positive semi-definite. The eigenvalues of $K(t)$ are related to the convergence rate of the training dynamics. Since it is difficult to directly compute all eigenvalues of $K(t)$, we instead consider the average of the eigenvalues $\frac{\operatorname{Tr}(K(t))}{N_r+N_b}$. It is a weighted average of $\frac{\operatorname{Tr}(K_{uu}(t))}{N_b}$ and $\frac{\operatorname{Tr}(K_{rr}(t))}{N_r}$ as follow:
\[\frac{\operatorname{Tr}(K(t))}{N_r+N_b}=\frac{N_b}{N_r+N_b}\frac{\operatorname{Tr}(K_{uu}(t))}{N_b}+\frac{N_r}{N_r+N_b}\frac{\operatorname{Tr}(K_{rr}(t))}{N_r}.\]

For notiational simplicity, we set the translation parameters $b_x=b_y=0$ and omit the hat ( $\hat{}$ ) notation in this subsection, thereby,
\[\mathcal{D}=\beta\frac{xy}{N_y}\partial_x + x(\partial_x)^+ +(\gamma-\beta\frac{x}{N_x})y\partial_y,
\]
and let our neural network be
\begin{equation}\label{eq:onelayer}
u(x,y;\theta)=\frac{1}{\sqrt{d_1}}\sum_{k=1}^{d_1}W_2^k \sigma(W_1^{1k}x+W_1^{2k}y+b_1^k)+b_2
\end{equation}
with $\sigma=\max\{0,x^3\}$ since the activation function is required to be twice differentiable in NTK theory ~\cite{jacot2018neural,wang2022and} to ensure stationarity of NTK.
We initialize the parameters $W_1^{ik}, W_2^k, b_1^k, b_2$ from the standard normal distribution $\mathcal{N}(0,1)$. 

We focus on training PINN within the rectangular domain 
\begin{equation}\label{domain-restricted}
[0,N_x\ell_x]\times [N_y\mu, N_y(\mu+\ell_y)],
\end{equation}
and introduce an augmented boundary given by
\begin{equation}\label{extended-boundary}
[0, N_x\ell_x]\times\{N_y \mu\}\cup \{0\}\times [N_y\mu, N_y(\mu+\ell_y)]\cup [0,N_x\ell_x]\times\{N_y(\mu+\ell_y)\}.
\end{equation}
Although this approach necessitates additional simulations to obtain ground-truth data for the augmented boundary given as~\eqref{extended-boundary}, the associated computational cost is significantly lower compared to simulating the entire two-dimensional domain. To proceed we sample the boundary points $\{x_b^j, y_b^j\}_{j=1}^{N_b}$ uniformly from the augmented boundary \eqref{extended-boundary}
and sample the residual points $\{x_r^i, y_r^i\}_{i=1}^{N_r}$ uniformly from the restricted domain \eqref{domain-restricted}.
Note that the domain \eqref{domain-restricted}
is obtained by scaling the domain $[0,\ell_x] \times [\mu,\mu+\ell_y]$ ($\ell_x,\ell_y>0$).

\subsubsection{Convergence of NTK for HJB}
We now present the main result, which demonstrates that the average eigenvalues of the deterministic kernel grows in the scaling factors.
\begin{thm}\label{thm:ntk}
Under the sampling regime of data points and the differential operator as described above, the following convergences hold as the width $d_1$ of \eqref{eq:onelayer} goes to infinity:
\[\frac{\operatorname{Tr}(K_{uu}(0))}{N_b}\overset{\mathcal{P}}{\rightarrow} O\left((N_x^2+N_y^2+1)^3\right)\quad\text{and}\quad\frac{\operatorname{Tr}(K_{rr}(0))}{N_r}\overset{\mathcal{P}}{\rightarrow} O\left(P(N_x, N_y)\right),
\]
where $\overset{\mathcal{P}}{\rightarrow}$ represents the convergence in probability and $P(x,y)$ is a degree $3$ homogeneous polynomial in $x^2$ and $y^2$ with positive coefficients.
Moreover, if $N_x=N_y=N$, then 
\[\frac{\operatorname{Tr}(K(0))}{N_r+N_b}\overset{\mathcal{P}}{\rightarrow} O(N^6)\footnote{\(f(x) = O(g(x))\) implies that there exist constants \(C > 0\) and \(x_0\) such that \(|f(x)| \leq C |g(x)|\) for all \(x \geq x_0\).}.\]
\end{thm}
Before presenting the proof, we state a useful lemma regarding the moment bounds of Gaussian distributions, which will be used repeatedly throughout the proof. The proof of the lemma is omitted and can be found in~\cite{hogg2013introduction}.

\begin{lem}\label{lem:gaussian}
    If \( X \sim \mathcal{N}(0, \delta^2) \), then 
\[
\mathbb{E}[X^{2n}] = \frac{(2n)!}{2^n n!} \delta^{2n}
\quad \text{and} \quad
\mathbb{E}[X^{2n-1}] = 0 \quad \forall n \in \mathbb{N}.
\]
\end{lem}

\begin{proof}[proof of Theorem~\ref{thm:ntk}]

We first compute the limit of $\operatorname{Tr}(K_{uu}(0))/N_b$ as the width $d_1$ goes to infinity.
From the definition, we see that
\[\frac{\operatorname{Tr}(K_{uu}(0))}{N_b}=\frac{1}{N_b}\sum_{j=1}^{N_b}\left\langle\frac{du(x_b^j,y_b^j;\theta)}{d\theta},\frac{du(x_b^j,y_b^j;\theta)}{d\theta} \right\rangle,\]
and each summand is decomposed as
\begin{align}\label{eq:kernel-decomposition}
    &\left\langle\frac{du(x_b^j,y_b^j;\theta)}{d\theta},\frac{du(x_b^j,y_b^j;\theta)}{d\theta} \right\rangle = \left\langle\frac{du(x_b^j,y_b^j;\theta)}{dW_1^1},\frac{du(x_b^j,y_b^j;\theta)}{dW_1^1} \right\rangle+\left\langle\frac{du(x_b^j,y_b^j;\theta)}{dW_1^2},\frac{du(x_b^j,y_b^j;\theta)}{dW_1^2}\right\rangle\\
    &+\left\langle\frac{du(x_b^j,y_b^j;\theta)}{dW_2},\frac{du(x_b^j,y_b^j;\theta)}{dW_2}\right\rangle
    + \left\langle\frac{du(x_b^j,y_b^j;\theta)}{db_1},\frac{du(x_b^j,y_b^j;\theta)}{db_1}\right\rangle+
    \left\langle\frac{du(x_b^j,y_b^j;\theta)}{db_2},\frac{du(x_b^j,y_b^j;\theta)}{db_2}\right\rangle.\notag
\end{align}
For notational convenience, we write
\[X_1:=W_1^{1k}\sim \mathcal{N}(0,1),\quad Y_1:=W_1^{2k}\sim \mathcal{N}(0,1),\quad X_2:=W_2^{k}\sim \mathcal{N}(0,1),\quad Z:=b_1^k\sim \mathcal{N}(0,1).\]
Since $X_1, Y_1,$ and $Z$ are independent, we have that
\[S:=X_1 x_b^j+Y_1 y_b^j+Z\sim\mathcal{N}(0, (x_b^j)^2+(y_b^j)^2+1).\]

The first term of \eqref{eq:kernel-decomposition} is then expressed as
\[\left\langle\frac{du(x_b^j,y_b^j;\theta)}{d\theta},\frac{du(x_b^j,y_b^j;\theta)}{d\theta} \right\rangle=\frac{1}{d_1}\sum_{k=1}^{d_1}\left(W_2^k\sigma'(W_1^{1k}x_b^j+W_1^{2k}y_b^j+b_1^k)x_b^j\right)^2.\]
By the law of large numbers, this term converges in probability to
\[(x_b^j)^2\mathbb{E}[X_2^2\sigma'(S)^2]=(x_b^j)^2\mathbb{E}[X_2^2]\mathbb{E}[\sigma'(S)^2]=(x_b^j)^2\mathbb{E}[\sigma'(S)^2]=(x_b^j)^2 O(((x_b^j)^2+(y_b^j)^2+1)^2)\]
as $d_1$ goes to infinity,
where the first equality holds since $X_2$ and $S$ are independent and the third equality follows from Lemma~\ref{lem:gaussian}.

The second term of \eqref{eq:kernel-decomposition} is 
\[\left\langle\frac{du(x_b^j,y_b^j;\theta)}{dW_1^2},\frac{du(x_b^j,y_b^j;\theta)}{dW_1^2}\right\rangle=\frac{1}{d_1}\sum_{k=1}^{d_1}\left(W_2^k\sigma(W_1^{1k}x_b^j+W_1^{2k}y_b^j+b_1^k)y_b^j\right)^2.\]
By the law of large numbers and the similar arguments, this term converges in probability to
\[(y_b^j)^2\mathbb{E}[X_2^2\sigma'(S)^2]=(y_b^j)^2\mathbb{E}[X_2^2]\mathbb{E}[\sigma'(S)^2]=(y_b^j)^2\mathbb{E}[\sigma'(S)^2]=(y_b^j)^2 O(((x_b^j)^2+(y_b^j)^2+1)^2)\]
as $d_1$ goes to infinity.

The third term of \eqref{eq:kernel-decomposition} is
\[\left\langle\frac{du(x_b^j,y_b^j;\theta)}{dW_2},\frac{du(x_b^j,y_b^j;\theta)}{dW_2}\right\rangle=\frac{1}{d_1}\sum_{k=1}^{d_1}\left(\sigma(W_1^{1k}x_b^j+W_1^{2k}y_b^j+b_1^k)\right)^2.\]
By the law of large numbers and Lemma~\ref{lem:gaussian}, this term converges in probability to
\[\mathbb{E}[\sigma(S)^2]=O(((x_b^j)^2+(y_b^j)^2+1)^3)\]
as $d_1$ goes to infinity.

The fourth term of \eqref{eq:kernel-decomposition} is 
\[\left\langle\frac{du(x_b^j,y_b^j;\theta)}{db_1},\frac{du(x_b^j,y_b^j;\theta)}{db_1}\right\rangle=\frac{1}{d_1}\sum_{k=1}^{d_1}\left(W_2^k\sigma'(W_1^{1k}x_b^j+W_1^{2k}y_b^j+b_1^k)\right)^2.\]
By the law of large numbers and Lemma~\ref{lem:gaussian}, this term converges in probability to
\[\mathbb{E}[\sigma'(S)^2]=O(((x_b^j)^2+(y_b^j)^2+1)^2)\]
as $d_1$ goes to infinity.

Finally, the last term in~\eqref{eq:kernel-decomposition} is
\[\left\langle\frac{du(x_b^j,y_b^j;\theta)}{db_2},\frac{du(x_b^j,y_b^j;\theta)}{db_2}\right\rangle=1.\]

Combining them all together, we deduce that
\[\left\langle\frac{du(x_b^j,y_b^j;\theta)}{d\theta},\frac{du(x_b^j,y_b^j;\theta)}{d\theta}\right\rangle=O\left(((x_b^j)^2+(y_b^j)^2+1)^3\right).\]
Noting we sample $(x_b^j, y_b^j)$ uniformly from the scaled boundary \eqref{extended-boundary}, $x_b^j$ and $y_b^j$ are scaled with $N_x$ and $N_y$ respectively. Thus, we conclude that 
\begin{equation}\label{eq:1}
\frac{\operatorname{Tr}(K_{uu}(0))}{N_b}=\frac{1}{N_b}\sum_{j=1}^{N_b}\left\langle\frac{du(x_b^j,y_b^j;\theta)}{d\theta},\frac{du(x_b^j,y_b^j;\theta)}{d\theta}\right\rangle=O\left((N_x^2+N_y^2+1)^3\right).
\end{equation}

Next, we compute the limit of $\operatorname{Tr}(K_{rr}(0))/N_r$ as the width $d_1$ goes to infinity. From the definition,
\[\frac{\operatorname{Tr}(K_{rr}(0))}{N_r}=\frac{1}{N_r}\sum_{i=1}^{N_r}\left\langle \frac{d\mathcal{D}[u(x_r^i,y_r^i;\theta)]}{d\theta},\frac{d\mathcal{D}[u(x_r^i,y_r^i;\theta)]}{d\theta}\right\rangle\]
and each summand is decomposed as
\begin{equation}\label{eq:kernel-decomposition2}
\begin{split}
    &\left\langle\frac{d\mathcal{D}[u(x_r^i,y_r^i;\theta)]}{d\theta},\frac{d\mathcal{D}[u(x_r^i,y_r^i;\theta)]}{d\theta} \right\rangle\\
    &\quad= \left\langle\frac{d\mathcal{D}[u(x_r^i,y_r^i;\theta)]}{dW_1^1},\frac{d\mathcal{D}[u(x_r^i,y_r^i;\theta)]}{dW_1^1} \right\rangle
    +\left\langle\frac{d\mathcal{D}[u(x_r^i,y_r^i;\theta)]}{dW_1^2},\frac{d\mathcal{D}[u(x_r^i,y_r^i;\theta)]}{dW_1^2}\right\rangle \\
    &\quad +\left\langle\frac{d\mathcal{D}[u(x_r^i,y_r^i;\theta)]}{dW_2},\frac{d\mathcal{D}[u(x_r^i,y_r^i;\theta)]}{dW_2}\right\rangle
    +\left\langle\frac{d\mathcal{D}[u(x_r^i,y_r^i;\theta)]}{db_1},\frac{d\mathcal{D}[u(x_r^i,y_r^i;\theta)]}{db_1}\right\rangle\\
    &\quad+
    \left\langle\frac{d\mathcal{D}[u(x_r^i,y_r^i;\theta)]}{db_2},\frac{d\mathcal{D}[u(x_r^i,y_r^i;\theta)]}{db_2}\right\rangle.
\end{split}
\end{equation}
Recall that 
\[\mathcal{D}[u]=\beta\frac{xy}{N_y}\partial_x u+x(\partial_x u)^++(\gamma-\beta\frac{x}{N_x})y\partial_y u.
\]
Computing the first derivatives of $u$, we get
\begin{align*}
    \partial_x u(x_r^i, y_r^i;\theta)&=\frac{1}{\sqrt{d_1}}\sum_{k=1}^{d_1} W_2^k W_1^{1k}\sigma'(W_1^{1k}x_r^i+W_1^{2k}y_r^i+b_1^k)
\end{align*}
and
\begin{align*}
    \partial_y u(x_r^i, y_r^i;\theta)&=\frac{1}{\sqrt{d_1}}\sum_{k=1}^{d_1} W_2^k W_1^{2k}\sigma'(W_1^{1k}x_r^i+W_1^{2k}y_r^i+b_1^k).
\end{align*}
Thus, we get
\begin{align*}
    \mathcal{D}[u(x_r^i,y_r^i;\theta)]&=
    \beta\frac{x_r^i y_r^i}{N_y\sqrt{d_1}}\sum_{k=1}^{d_1}W_2^k W_1^{1k}\sigma'(W_1^{1k}x_r^i+W_1^{2k}y_r^i+b_1^k)\\
    &+\frac{x_r^i}{\sqrt{d_1}}\left(\sum_{k=1}^{d_1}W_2^k W_1^{1k}\sigma'(W_1^{1k}x_r^i+W_1^{2k}y_r^i+b_1^k)\right)^+\\
    &+(\gamma-\beta\frac{x_r^i}{N_x})\frac{y_r^i}{\sqrt{d_1}}\sum_{k=1}^{d_1}W_2^k W_1^{2k}\sigma'(W_1^{1k}x_r^i+W_1^{2k}y_r^i+b_1^k).
\end{align*}
The first term of \eqref{eq:kernel-decomposition2} is
\begin{align*}
    &\left\langle\frac{d\mathcal{D}[u(x_r^i,y_r^i;\theta)]}{dW_1^1},\frac{d\mathcal{D}[u(x_r^i,y_r^i;\theta)]}{dW_1^1} \right\rangle\\
    &=
    \frac{1}{d_1}\sum_{k=1}^{d_1}\bigg[\frac{\beta x_r^i y_r^i}{N_y}W_2^k \left(W_1^{1k}\sigma''(W_1^{1k}x_r^i+W_1^{2k}y_r^i+b_1^k)x_r^i+\sigma'(W_1^{1k}x_r^i+W_1^{2k}y_r^i+b_1^k)\right)\\
    &\quad+\delta x_r^i W_2^k\left(W_1^{1k}\sigma''(W_1^{1k}x_r^i+W_1^{2k}y_r^i+b_1^k)x_r^i+\sigma'(W_1^{1k}x_r^i+W_1^{2k}y_r^i+b_1^k)\right)\\
    &\quad+(\gamma-\beta\frac{x_r^i}{N_x})y_r^i W_2^k W_1^{1k}\sigma''(W_1^{1k}x_r^i+W_1^{2k}y_r^i+b_1^k)x_r^i\bigg]^2,
\end{align*}
where 
\[
\delta=\begin{cases}
    1 & \text{ if }\sum_{k=1}^{d_1}W_2^k W_1^{1k}\sigma'(W_1^{1k}x_r^i+W_1^{2k}y_r^i+b_1^k)\geq 0,\\
    0 & \text{ if }\sum_{k=1}^{d_1}W_2^k W_1^{1k}\sigma'(W_1^{1k}x_r^i+W_1^{2k}y_r^i+b_1^k)<0.
\end{cases}
\]
Temporarily denoting
\begin{align*}
    A_k &= \frac{\beta x_r^i y_r^i}{N_y}W_2^k W_1^{1k}\sigma''(W_1^{1k}x_r^i+W_1^{2k}y_r^i+b_1^k)x_r^i,\\
    B_k &= \frac{\beta x_r^i y_r^i}{N_y}W_2^k \sigma'(W_1^{1k}x_r^i+W_1^{2k}y_r^i+b_1^k),\\
    C_k &= x_r^i W_2^k W_1^{1k}\sigma''(W_1^{1k}x_r^i+W_1^{2k}y_r^i+b_1^k)x_r^i,\\
    D_k &= x_r^i W_2^k \sigma'(W_1^{1k}x_r^i+W_1^{2k}y_r^i+b_1^k),\\
    E_k &= (\gamma-\beta\frac{x_r^i}{N_x})y_r^i W_2^k W_1^{1k}\sigma''(W_1^{1k}x_r^i+W_1^{2k}y_r^i+b_1^k)x_r^i,
\end{align*}
and applying the power-mean inequality, we achieve
\[ \left\langle\frac{d\mathcal{D}[u(x_r^i,y_r^i;\theta)]}{dW_1^1},\frac{d\mathcal{D}[u(x_r^i,y_r^i;\theta)]}{dW_1^1} \right\rangle\leq \frac{1}{d_1}\sum_{k=1}^{d_1}5(A_k^2+B_k^2+C_k^2+D_k^2+E_k^2).\]
By the law of large numbers,
\begin{align*}
    \frac{1}{d_1}\sum_{k=1}^{d_1} A_k^2&\overset{\mathcal{P}}{\rightarrow} \frac{\beta^2(x_r^i)^4(y_r^i)^2}{N_y^2}\mathbb{E}[X_1^2\sigma''(S)^2],\\
    \frac{1}{d_1}\sum_{k=1}^{d_1} B_k^2&\overset{\mathcal{P}}{\rightarrow}
    \frac{\beta^2(x_r^i)^2(y_r^i)^2}{N_y^2}\mathbb{E}[\sigma'(S)^2],\\
    \frac{1}{d_1}\sum_{k=1}^{d_1} C_k^2&\overset{\mathcal{P}}{\rightarrow}
    (x_r^i)^2\mathbb{E}[X_1^2\sigma''(S)^2],\\
    \frac{1}{d_1}\sum_{k=1}^{d_1} D_k^2&\overset{\mathcal{P}}{\rightarrow}
    (x_r^i)^2\mathbb{E}[\sigma'(S)^2],\\
    \frac{1}{d_1}\sum_{k=1}^{d_1} E_k^2&\overset{\mathcal{P}}{\rightarrow}
    (\gamma-\beta\frac{x_r^i}{N_x})^2 (x_r^i)^2(y_r^i)^2\mathbb{E}[X_1^2\sigma''(S)^2].
\end{align*}
where $S=X_1x_r^i+Y_1 y_r^i+Z$ and $\overset{\mathcal{P}}{\rightarrow}$ denotes the convergence in probability. By Lemma~\ref{lem:gaussian}, we know that
$\mathbb{E}[\sigma'(S)^2]=O\left(((x_r^i)^2+(y_r^i)^2+1)^2\right).$ Using the fact 
\[|\sigma''(X_1 x_r^i +Y_1 y_r^i +Z)|^2 \leq 36 |X_1 x_r^i + Y_1 y_r^i + Z|^2,
\] 
we have
\begin{equation}\label{eq:computation}
\begin{split}
&\mathbb{E}\left[ X_1^2 \sigma''(S)^2 \right] \\
&\leq \frac{36}{(2\pi)^{3/2}} \int_{-\infty}^\infty \int_{-\infty}^\infty 
x^2(x^2(x_r^i)^2+y^2(y_r^i)^2+z^2+2xyx_r^i y_r^i+2yzy_r^i+2xzx_r^i)\exp\left(-\frac{x^2+y^2+z^2}{2}\right) \mathrm{d}z \mathrm{d}dy \mathrm{d}dx\\
&= 36 (|x_r^i|^2\mathbb{E}[X_1^4]+(y_r^i)^2\mathbb{E}[X_1^2 Y_1^2]+\mathbb{E}[X_1^2 Z^2])\\
&= 36 \left(3 (x_r^i)^2 + (y_r^i)^2 + 1\right).
\end{split}
\end{equation}
Since we sample $(x_r^i, y_r^i)$ from the scaled domain \eqref{domain-restricted}, $x_r^i$ and $y_r^i$ are scaled with $N_x$ and $N_y$ respectively, as we adjust $N_x, N_y$. Thus, we have that
\begin{equation}\label{eq:2-1}
\left\langle\frac{d\mathcal{D}[u(x_r^i,y_r^i;\theta)]}{dW_1^1},\frac{d\mathcal{D}[u(x_r^i,y_r^i;\theta)]}{dW_1^1} \right\rangle=
O \left((N_x^4+N_x^2+N_x^2 N_y^2)(3N_x^2+N_y^2+1)+N_x^2(N_x^2+N_y^2+1)^2\right),
\end{equation}
where we omitted other variables such as $\beta$ and $\gamma$ as our primary focus is on the dependency on $N_x$ and $N_y$.

To proceed, we now observe that the second term in~\eqref{eq:kernel-decomposition2} is written as
\begin{align*}
    &\left\langle\frac{d\mathcal{D}[u(x_r^i,y_r^i;\theta)]}{dW_1^2},\frac{d\mathcal{D}[u(x_r^i,y_r^i;\theta)]}{dW_1^2} \right\rangle\\
    &=
    \frac{1}{d_1}\sum_{k=1}^{d_1}\bigg[\frac{\beta x_r^i y_r^i}{N_y}W_2^k W_1^{1k}\sigma''(W_1^{1k}x_r^i+W_1^{2k}y_r^i+b_1^k)y_r^i\\
    &\quad+\delta x_r^i W_2^kW_1^{1k}\sigma''(W_1^{1k}x_r^i+W_1^{2k}y_r^i+b_1^k)y_r^i\\
    &\quad+(\gamma-\beta\frac{x_r^i}{N_x})y_r^i W_2^k\left( W_1^{1k}\sigma''(W_1^{1k}x_r^i+W_1^{2k}y_r^i+b_1^k)y_r^i+\sigma'(W_1^{1k}x_r^i+W_1^{2k}y_r^i+b_1^k)\right)\bigg]^2.
\end{align*}
Similarly we temporarily denote
\begin{align*}
    A_k &= \frac{\beta x_r^i y_r^i}{N_y}W_2^k W_1^{1k}\sigma''(W_1^{1k}x_r^i+W_1^{2k}y_r^i+b_1^k)y_r^i,\\
    B_k &= x_r^i W_2^k W_1^{1k}\sigma''(W_1^{1k}x_r^i+W_1^{2k}y_r^i+b_1^k)y_r^i,\\
    C_k &= (\gamma-\beta\frac{x_r^i}{N_x})y_r^i W_2^k W_1^{2k}\sigma''(W_1^{1k}x_r^i+W_1^{2k}y_r^i+b_1^k)y_r^i,\\
    D_k &= (\gamma-\beta\frac{x_r^i}{N_x})y_r^i W_2^k \sigma'(W_1^{1k}x_r^i+W_1^{2k}y_r^i+b_1^k)y_r^i,
\end{align*}
and apply the power mean inequality to deduce
\[ \left\langle\frac{d\mathcal{D}[u(x_r^i,y_r^i;\theta)]}{dW_1^2},\frac{d\mathcal{D}[u(x_r^i,y_r^i;\theta)]}{dW_1^2} \right\rangle\leq \frac{1}{d_1}\sum_{k=1}^{d_1}4(A_k^2+B_k^2+C_k^2+D_k^2).\]
By the law of large numbers,
\begin{align*}
    \frac{1}{d_1}\sum_{k=1}^{d_1} A_k^2&\overset{\mathcal{P}}{\rightarrow} \frac{\beta^2(x_r^i)^2(y_r^i)^4}{N_y^2}\mathbb{E}[X_1^2\sigma''(S)^2],\\
    \frac{1}{d_1}\sum_{k=1}^{d_1} B_k^2&\overset{\mathcal{P}}{\rightarrow}
    (x_r^i)^2(y_r^i)^2\mathbb{E}[X_1^2\sigma''(S)^2],\\
    \frac{1}{d_1}\sum_{k=1}^{d_1} C_k^2&\overset{\mathcal{P}}{\rightarrow}
    (\gamma-\beta\frac{x_r^i}{N_x})^2(y_r^i)^4\mathbb{E}[Y_1^2\sigma''(S)^2],\\
    \frac{1}{d_1}\sum_{k=1}^{d_1} D_k^2&\overset{\mathcal{P}}{\rightarrow}
    (\gamma-\beta\frac{x_r^i}{N_x})^2(y_r^i)^2\mathbb{E}[\sigma'(S)^2].
\end{align*}
Thus, by Lemma~\ref{lem:gaussian} and \eqref{eq:computation}, we get
\begin{equation}\label{eq:2-2}
\left\langle\frac{d\mathcal{D}[u(x_r^i,y_r^i;\theta)]}{dW_1^2},\frac{d\mathcal{D}[u(x_r^i,y_r^i;\theta)]}{dW_1^2} \right\rangle=
O \left((N_x^2 N_y^2+N_y^4)(3N_x^2+N_y^2+1)+N_y^2(N_x^2+N_y^2+1)^2\right).
\end{equation}

The third term of \eqref{eq:kernel-decomposition2} is
\begin{align*}
\left\langle\frac{d\mathcal{D}[u(x_r^i,y_r^i;\theta)]}{dW_2},\frac{d\mathcal{D}[u(x_r^i,y_r^i;\theta)]}{dW_2} \right\rangle &=
\frac{1}{d_1}\sum_{k=1}^{d_1}\bigg[\frac{\beta x_r^i y_r^i}{N_y}W_1^{1k}\sigma'(W_1^{1k}x_r^i+W_1^{2k}y_r^i+b_1^k)\\
&+\delta x_r^i W_1^{1k}\sigma'(W_1^{1k}x_r^i+W_1^{2k}y_r^i+b_1^k)\\
&+(\gamma-\beta\frac{x_r^i}{N_x})y_r^i W_1^{2k}\sigma'(W_1^{1k}x_r^i+W_1^{2k}y_r^i+b_1^k)\bigg]^2.
\end{align*}
Denoting
\begin{align*}
    A_k &= \frac{\beta x_r^i y_r^i}{N_y}W_1^{1k}\sigma'(W_1^{1k}x_r^i+W_1^{2k}y_r^i+b_1^k),\\
    B_k &= x_r^i W_1^{1k}\sigma'(W_1^{1k}x_r^i+W_1^{2k}y_r^i+b_1^k),\\
    C_k &= (\gamma-\beta\frac{x_r^i}{N_x})y_r^i W_1^{2k}\sigma'(W_1^{1k}x_r^i+W_1^{2k}y_r^i+b_1^k),
\end{align*}
we have 
\[\left\langle\frac{d\mathcal{D}[u(x_r^i,y_r^i;\theta)]}{dW_2},\frac{d\mathcal{D}[u(x_r^i,y_r^i;\theta)]}{dW_2} \right\rangle\leq \frac{1}{d_1}\sum_{k=1}^{d_1}3(A_k^2+B_k^2+C_k^2).\]
By the law of large numbers,
\begin{align*}
    \frac{1}{d_1}\sum_{k=1}^{d_1} A_k^2&\overset{\mathcal{P}}{\rightarrow} \frac{\beta^2(x_r^i)^2(y_r^i)^2}{N_y^2}\mathbb{E}[X_1^2\sigma'(S)^2],\\
    \frac{1}{d_1}\sum_{k=1}^{d_1} B_k^2&\overset{\mathcal{P}}{\rightarrow}
    (x_r^i)^2\mathbb{E}[X_1^2\sigma'(S)^2],\\
    \frac{1}{d_1}\sum_{k=1}^{d_1} C_k^2&\overset{\mathcal{P}}{\rightarrow}
    (\gamma-\beta\frac{x_r^i}{N_x})^2(y_r^i)^2\mathbb{E}[Y_1^2\sigma'(S)^2].
\end{align*}
By Lemma~\ref{lem:gaussian} and the fact that $\sigma'(x)^2\leq 9x^4$, we have
\begin{align*}
&\mathbb{E}[X_1^2\sigma'(S)^2]\\
&\leq\frac{9}{(2\pi)^{3/2}}\int_{-\infty}^\infty\int_{-\infty}^\infty\int_{-\infty}^\infty x^2(x x_r^i+yy_r^i+z)^4 \mathrm{d}z\mathrm{d}y\mathrm{d}x\\
&= \frac{9}{(2\pi)^{3/2}}\int_{-\infty}^\infty\int_{-\infty}^\infty\int_{-\infty}^\infty x^2(x^4 (x_r^i)^4+y^4(y_r^i)^4+z^4+6x^2y^2(x_r^i y_r^i)^2+6y^2z^2(y_r^i)^2+6x^2z^2(x_r^i)^2) \mathrm{d}z\mathrm{d}y\mathrm{d}x\\
&= 9(15(x_r^i)^4+3(y_r^i)^4+3+18(x_r^i y_r^i)^2+6(y_r^i)^2+18(x_r^i)^2)=O(5(x_r^i)^4+(y_r^i)^4+6(x_r^i y_r^i)^2).
\end{align*}
Similarly, 
\[\mathbb{E}[Y_1^2\sigma'(S)^2]\leq 9(15(y_r^i)^4+3(x_r^i)^4+3+18(x_r^i y_r^i)^2+6(x_r^i)^2+18(y_r^i)^2)=O((x_r^i)^4+5(y_r^i)^4+6(x_r^i y_r^i)^2).\]
Thus, we have that
\begin{equation}\label{eq:2-3}
\begin{split}
    \left\langle\frac{d\mathcal{D}[u(x_r^i,y_r^i;\theta)]}{dW_2},\frac{d\mathcal{D}[u(x_r^i,y_r^i;\theta)]}{dW_2} \right\rangle&=O\left(N_x^2(5N_x^4+N_y^4+6N_x^2 N_y^2)+N_y^2(N_x^4+5N_y^4+6N_x^2 N_y^2)\right)\\
    &=O(5N_x^6+7N_x^4N_y^2+7N_x^2 N_y^4+5N_y^6).
    \end{split}
\end{equation}

The fourth term of \eqref{eq:kernel-decomposition2} is
\begin{align*}
    \left\langle\frac{d\mathcal{D}[u(x_r^i,y_r^i;\theta)]}{db_1},\frac{d\mathcal{D}[u(x_r^i,y_r^i;\theta)]}{db_1} \right\rangle
    &=\frac{1}{d_1}\sum_{k=1}^{d_1}\bigg[\frac{\beta x_r^i y_r^i}{N_y}W_2^k W_1^{1k}\sigma''(W_1^{1k}x_r^i+W_1^{2k}y_r^i+b_1^k)\\
    &+\delta x_r^i W_2^k W_1^{1k}\sigma''(W_1^{1k}x_r^i+W_1^{2k}y_r^i+b_1^k)\\
    &+(\gamma-\beta\frac{x_r^i}{N_x})y_r^i W_2^k W_1^{2k}\sigma''(W_1^{1k}x_r^i+W_1^{2k}y_r^i+b_1^k)\bigg]^2.
\end{align*}
Denoting
\begin{align*}
    A_k &= \frac{\beta x_r^i y_r^i}{N_y}W_2^k W_1^{1k}\sigma''(W_1^{1k}x_r^i+W_1^{2k}y_r^i+b_1^k),\\
    B_k &= x_r^i W_2^k W_1^{1k}\sigma''(W_1^{1k}x_r^i+W_1^{2k}y_r^i+b_1^k),\\
    C_k &= (\gamma-\beta\frac{x_r^i}{N_x})y_r^i W_2^k W_1^{2k}\sigma''(W_1^{1k}x_r^i+W_1^{2k}y_r^i+b_1^k),
\end{align*}
we have 
\[\left\langle\frac{d\mathcal{D}[u(x_r^i,y_r^i;\theta)]}{db_1},\frac{d\mathcal{D}[u(x_r^i,y_r^i;\theta)]}{db_1} \right\rangle\leq \frac{1}{d_1}\sum_{k=1}^{d_1}3(A_k^2+B_k^2+C_k^2).\]
By the law of large numbers,
\begin{align*}
    \frac{1}{d_1}\sum_{k=1}^{d_1} A_k^2&\overset{\mathcal{P}}{\rightarrow} \frac{\beta^2(x_r^i)^2(y_r^i)^2}{N_y^2}\mathbb{E}[X_1^2\sigma''(S)^2],\\
    \frac{1}{d_1}\sum_{k=1}^{d_1} B_k^2&\overset{\mathcal{P}}{\rightarrow}
    (x_r^i)^2\mathbb{E}[X_1^2\sigma''(S)^2],\\
    \frac{1}{d_1}\sum_{k=1}^{d_1} C_k^2&\overset{\mathcal{P}}{\rightarrow}
    (\gamma-\beta\frac{x_r^i}{N_x})^2(y_r^i)^2\mathbb{E}[Y_1^2\sigma''(S)^2].
\end{align*}
By \eqref{eq:computation}, we have
\begin{equation}\label{eq:2-4}
\left\langle\frac{d\mathcal{D}[u(x_r^i,y_r^i;\theta)]}{db_1},\frac{d\mathcal{D}[u(x_r^i,y_r^i;\theta)]}{db_1} \right\rangle=O\left(N_x^2(3N_x^2+N_y^2+1)+N_y^2(N_x^2+3N_y^2+1)\right).
\end{equation}

Lastly, the fifth term of \eqref{eq:kernel-decomposition2} is simply
\begin{equation}\label{eq:2-5}
    \left\langle\frac{d\mathcal{D}[u(x_r^i,y_r^i;\theta)]}{db_2},\frac{d\mathcal{D}[u(x_r^i,y_r^i;\theta)]}{db_2} \right\rangle=0.
\end{equation}

Combining \eqref{eq:2-1}, \eqref{eq:2-2}, \eqref{eq:2-3}, \eqref{eq:2-4} and \eqref{eq:2-5}, we obtain that
\begin{equation}\label{eq:2}
\frac{\operatorname{Tr}(K_{rr}(0))}{N_r}=O\left(P(N_x, N_y)\right),
\end{equation}
where $P(x,y)$ is a degree $3$ homogeneous polynomial in $x^2$ and $y^2$ with positive coefficients.

By \eqref{eq:1} and \eqref{eq:2}, we finally conclude that the average convergence rate of NTK increases as the variable scaling factors $N_x$ and $N_y$ increase.
In the special case where we set $N_x=N_y=N$, then the right hand sides of both \eqref{eq:1} and \eqref{eq:2} become $O(N^6)$.
\end{proof}

In practice, we use gradient descent in training, which is not identical to the ideal gradient flow. Consequently, arbitrarily increasing $N_x$ and $N_y$ may lead to excessively large parameter updates during a single step of gradient descent, potentially destabilizing the training process. Therefore, it is crucial to empirically determine an appropriate value for $N_x$ and $N_y$ to ensure stable learning.

\section{Methodology}\label{sec:method}
Our goal is to approximate $u$ and $u^{r_0}$ using the PINN framework without discretizing the spatial domain and to use these results to estimate the optimal switching control via~\eqref{eq:switch} with high accuracy. 
To achieve this, we train \(u\) and \(u^{r_0}\) separately by solving~\eqref{eq:problem_main} within a fixed domain \(D := [0, \ell_x] \times [\mu, \mu + \ell_y]\), where \(\ell_x, \ell_y, \mu > 0\) are given. While the original domain is the unbounded, restriction to the bounded domain is a common practice in numerical analysis. We then identify the optimal control $r$ via~\eqref{eq:switch}, which minimizes the eradication time, by modeling $\tau$ as a neural network function. For the ease of training the switching time $\tau$, we introduce another neural network to approximate the uncontrolled dynamics $(S(t),I(t))$ satisfying~\eqref{eq:ode_uncontrolled}. 

Figure~\ref{fig:tau_problem_solving} outlines the full workflow, where we first train \(u\) and \(u^{r_0}\) using PINNs, then learn the uncontrolled system dynamics, and finally optimize \(\tau\) using the DPP framework.

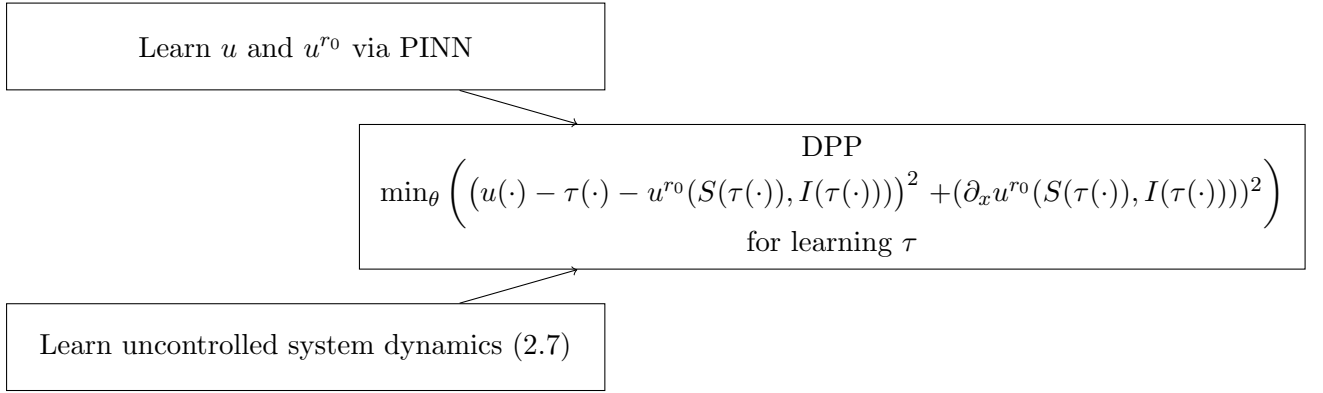
\begin{figure}[H]
    \centering
    \begin{tikzpicture}
      \tikzset{block/.style={rectangle, draw, text width=8em, text centered, minimum height=3em}}

      \node[block, text width=20em] (u) at (-3,4) {Learn $u$ and $u^{r_0}$ via PINN};
      
      \node[block, text width=20em] (trained) at (-3,0) {Learn uncontrolled system dynamics~\eqref{eq:ode_uncontrolled}};
      \node[block, text width=32em, minimum height=5em] (dpp) at (4,2) {DPP\\ $\min_{\theta} \bigg(\big(u(\cdot) - \tau(\cdot) - u^{r_0}(S(\tau(\cdot)), I(\tau(\cdot)))\big)^2$
      $ + (\partial_x u^{r_0}(S(\tau(\cdot)), I(\tau(\cdot))))^2\bigg)$\\ for learning $\tau$};
    
      \draw[->] (u) -- (dpp);
      \draw[->] (trained) -- (dpp);
    \end{tikzpicture}
    \caption{Procedure for solving \(\tau\) using the DPP framework, with trained \(u\), \(u^{r_0}\), and the learned uncontrolled dynamics \((S(t), I(t))\).}
    \label{fig:tau_problem_solving}
\end{figure}

\subsection{Solving HJB equations}
\label{sec:training_u}
Let us begin by solving 
\[
        \beta x y \partial_x u + x (\partial_x u)^+ +(\gamma - \beta x ) y \partial_y u=1\quad\text{in}\quad [0,\ell_x] \times [\mu,\mu+\ell_y],
\]
and 
\[
        \beta x y \partial_x u^{r_0} + x \partial_x u^{r_0} + (\gamma - \beta x ) y \partial_y \quad\text{in}\quad[0,\ell_x] \times [\mu,\mu+\ell_y],
\]
respectively via PINNs, or equivalently,
\[
        \mathcal{D}[u]=1\quad \text{and}\quad 
        \mathcal{D}^0[u^{r_0}]=1 \quad[0,\ell_x] \times [\mu,\mu+\ell_y],
\]
where 
\[
\mathcal{D} = \beta xy\partial_x + x(\partial_x)^+ +(\gamma-\beta x)y\partial_y \quad\text{and}\quad \mathcal{D}^0 = \beta xy\partial_x + x\partial_x +(\gamma-\beta x)y\partial_y.
\]

To leverage the idea of the VS-PINN, we set 
\[
u(x,y;\theta_1)=\text{NN}\bigg(N_x x+b_x,N_y y+b_y;\theta_1\bigg) \quad \text{and}\quad u^{r_0}(x,y;\theta_2)= \text{NN}\bigg(N_x x+b_x,N_y y+b_y;\theta_2\bigg),
\]
where $\text{NN}(\cdot;\theta_1)$ and $\text{NN}(\cdot;\theta_2)$ represent fully connected neural networks parametrized by $\theta_1$ and $\theta_2$ respectively. Note that $N_x, N_y, b_x, b_y$ are not trained during training. We then explicitly define the loss function.

\textbf{Residual loss}: With the prespecified training data $\{(x_r^i,y_r^i)\}_{i=1}^{N_r} \in [0,\ell_x] \times [\mu,\ell_y+\mu]$ for $N_r \in \mathbb{N}$,
the residual loss associated with $u$ is given as 
\[
\mathcal{L}_r[u(\cdot;\theta_1)] = \frac{1}{N_r}\sum_{i=1}^{N_r} (\mathcal{D}[u(x_r^i,y_r^i;\theta_1)]-1)^2.
\]

Similarly, the residual loss corresponding to $u^{r_0}$ is defined as
\[
\mathcal{L}^0_r[u^{r_0}(\cdot;\theta_2)] = \frac{1}{N_r}\sum_{i=1}^{N_r} (\mathcal{D}^0[u^{r_0}(x_r^i,y_r^i;\theta_2)]-1)^2.
\]

\textbf{Boundary loss}: We define boundary loss function on 
\begin{equation}
\label{def:bdry}
\Gamma:= \underbrace{[0,\ell_x] \times \{\mu\}}_{=:D_1} \cup \underbrace{\{0\} \times [\mu,\mu+\ell_y]}_{=:D_2} \cup \underbrace{[0,\ell_x] \times \{\mu+\ell_y\}}_{=:D_3} \subset \partial D,
\end{equation}
where $\partial D$ denotes the boundary of the domain $D$. Given $N_b^k \in \mathbb {N}$, we randomly sample $N_b^k$ points from $D_k$ and denote them by $\{(x_{b,k}^j,y_{b,k}^j)\}_{j=1}^{N_b^k}$ for $k=1,2,3$. With $N_b = N_b^1 + N_b^2 + N_b^3$, let
\begin{align}
\label{eq:bdry_loss}
    \mathcal{L}_b[u(\cdot;\theta)] &= \frac{1}{N_b}\sum_{k=1}^3\sum_{j=1}^{N_b^k}\left(u(x_{b,k}^j,y_{b,k}^j;\theta)-u(x_{b,k}^j,y_{b,k}^j)\right)^2.
\end{align}
We note that boundary loss for learning $u^{r_0}$ is identical to that used for learning $u$.

One key challenge is the approximation of \(u(x_{b,k}^j, y_{b,k}^j)\) at the boundary points for each \(j\) and \(k\). To address this, we revisit an intrinsic property of optimal control for the minimum eradication time, demonstrating that the optimal vaccination strategy takes the form of~\eqref{eq:bangbang}. Accordingly, we propose Algorithm~\ref{alg:approx_era} for numerical implementation. Leveraging the structure of the optimal control, we compute the first time when \(I\) falls below \(\mu\) by using controls \(r_{s d \tau}\) for \(s \in \mathbb{N} \cup \{0\}\) and \(d \tau > 0\). The solution to the controlled SIR model~\eqref{eq:controlled_raw} is approximated using the fourth-order Runge-Kutta method with a stepsize \(dt>0\).

\begin{algorithm}[h]
\caption{Find the minimum eradication time}
\label{alg:approx_era}
\begin{algorithmic}[1]
\State {\bf Input:} Time resolution of optimal control $d\tau$, maximum switching time $L$, the discretization size $dt$, maximum iteration step $M$, eradication threshold $\mu$, initial susceptible and infectious population $x,y\in D$.

\State {\bf Output:} $\tau>0$ such that $r_\tau$ is an optimal control.
\State Set $(S_0,I_0)=(x,y)$ and $\tau=0$.
\While{True}
\For{$m=0,1,2,\ldots,M$} 
\State Compute \( S_{m+1}, I_{m+1} \) using the Runge-Kutta method corresponding to
\begin{equation*}
\begin{cases}
\dot S&=-\beta S I - r_\tau S,\\
\dot I&= \beta S I - \gamma I.
\end{cases}
\end{equation*}
\EndFor
\State $m_\tau:=\min \{m' : I_{m'} \leq \mu\}$.
\State $\tau = \tau + d\tau$.
\State Break if $\tau>L$. 
\EndWhile
\State Return $u(x,y) = \min_\tau \{m_\tau\}$.
\end{algorithmic}
\end{algorithm}

\subsection{Optimal switching control}
With $u$ and $u^{r_0}$ computed in the previous section, we now solve for the optimal switching time $\tau$ satisfying~\eqref{eq:switch},
\begin{equation*}
u(x,y)=\min_{\tau \geq 0} \{\tau + u^{r_0}(S(\tau),I(\tau))\},
\end{equation*}
which involves a complex and nonlinear optimization process. To avoid solving this optimization problem directly, we introduce a neural network $\tau(x,y;\theta)$, and propose optimizing the following problem: 
\begin{equation*}
\min_{\theta} \bigg(u(\cdot) -  \tau(\cdot;\theta) - u^{r_0}(S( \tau(\cdot;\theta)),I( \tau(\cdot;\theta)))\bigg)^2,    
\end{equation*}
where $S$ and $I$ satisfy the uncontrolled system of ODE~\eqref{eq:ode_uncontrolled}, that is,
\begin{equation*}
\begin{cases}
\dot S &= -\beta S I,\\
\dot I &= \beta S I - \gamma I,
\end{cases}
\end{equation*}
with $(S(0),I(0))=(x,y)$. Since the closed-form solutions of $S$ and $I$ are unavailable, we reduce the above optimization problem into
\begin{equation*}
\min_{\theta} \bigg((u(\cdot) -  \tau(\cdot;\theta) - u^{r_0}( S( \tau(\cdot;\theta);\omega), I( \tau(\cdot;\theta);\omega)))^2\bigg),    
\end{equation*}
where we approximate $(S(t),I(t))$ using a neural network $w(x,y,t;\omega)$, that is,
\[
(S(t),I(t)) \approx w(x,y,t;\omega).
\]

This approximation corresponds to training a neural network $w(x,y,t;\omega)$ with $w(x,y,0;\omega)=(x,y)$ that best mimics the uncontrolled dynamics $(S(t),I(t))$ satisfying~\eqref{eq:ode_uncontrolled} with $(S(0),I(0))=(x,y)$.

\subsubsection{Learning the uncontrolled SIR model}
\label{sub:learning_dy}
\begin{figure}[ht]
    \centering
    \includegraphics[width=\linewidth]{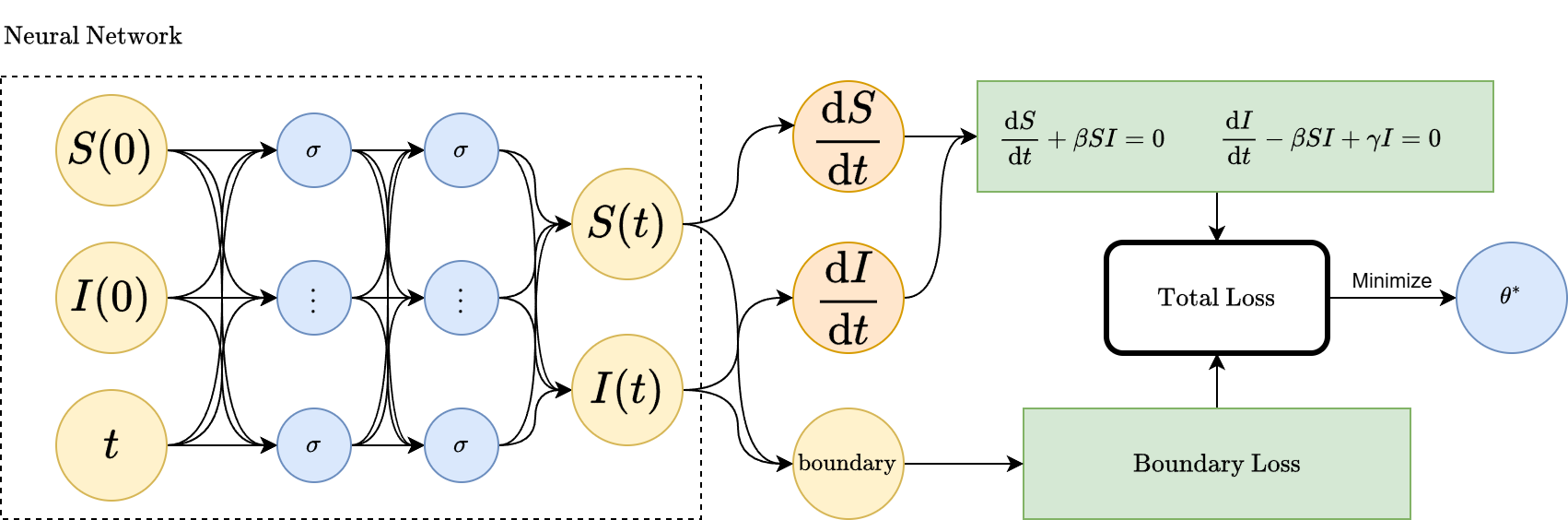}
    \caption{Training of uncontrolled dynamics via the PINN framework}
    \label{fig:si_pinn}
\end{figure}

\begin{algorithm}[h]
\caption{Training of the dynamics}
\label{alg:uncontolled_dynamics}
\begin{algorithmic}[1]
\State {\bf Input:} Number of training samples $N_{p}$, $N_{\text{int}}$, $N_{\text{bdry}}>0$, the number of temporal discretization points $N_T$, maximum evolution time $T$, the discretization step $dt=T/N_T>0$.
\State {\bf Output:} $w(x,y,t;\omega) \approx (S(t),I(t))$ solving ~\eqref{eq:ode_uncontrolled} with $(S(0),I(0))=(x,y)$.
\State Sample $\{x_{p}^i,y_{p}^i\}_{i=1}^{N_{p}} \subset D$.
\State Sample $
\{(x_{\text{bdry}}^k,y_{\text{bdry}}^k,t_k)\}_{k=1}^{N_{\text{bdry}}} \in \partial D \times \mathcal{T}$.
\While{Training not converged}
\State Sample spatial collocation points $\{x_{\text{int}}^j,y_{\text{int}}^j\}_{j=1}^{N_{\text{int}}} \subset D \times [0,T]$.
\State Set $(S_0,I_0)=(x,y)$.
\For{$m=0,1,2,\ldots,N_T-1$} 
\State 
Compute \( S_{m+1}, I_{m+1} \) using the Runge-Kutta method corresponding to
\begin{equation*}
\begin{cases}
\dot S &=-\beta S I ,\\
\dot I &= \beta S I - \gamma I.
\end{cases}
\end{equation*}
\EndFor
\State Compute the loss function $\mathcal{L}[w(\cdot;\omega)]$ defined as~\eqref{eq:loss_dy}.
\State $\omega \leftarrow \text{Adam}(\mathcal{L})$.
\EndWhile
\end{algorithmic}
\end{algorithm}

We propose to train a neural network $w(x,y,t;\theta) : (x,y,t) \mapsto \mathbb{R}^2$ that approximates the flow of uncontrolled dynamics $(S(t),I(t))$ with $(S(0),I(0))=(x,y)$ for all points $(x,y) \in D$. To train such a $w$, we first randomly choose interior points from the domain $D=[0,\ell_x] \times [\mu,\mu+\ell_y]$. For the stability of the training, we further sample points from the boundary of $D$ given by
\begin{equation*}
\label{eq:w_boundary}
\partial D= [0,\ell_x] \times \{\mu\} \cup \{0\} \times [\mu,\mu+\ell_y] \cup [0,\ell_x] \times \{\mu+\ell_y\} \cup \{\ell_x\} \times [\mu,\mu+\ell_y].
\end{equation*}
Let $T$, $N_T$, $N_p$, $N_{\text{int}}$, and $N_{\text{bdry}}$ be given and set $dt=T/N_T$ and $\mathcal{T}:=\{mdt:m=0,1,2,...,N_T\}$. With random samples $\{(x_{p}^i,y_{p}^i)\}_{i=1}^{N_{p}} \subset D$, $\{(x_{\text{int}}^j,y_{\text{int}}^j,t_j)\}_{j=1}^{N_{\text{int}}} \subset D \times [0,T]$, $\{(x_{\text{bdry}}^k,y_{\text{bdry}}^k,t_k)\}_{k=1}^{N_{\text{bdry}}} \subset \partial D \times \mathcal{T}$, and $N_{\text{tot}} = N_p + N_{\text{int}}+N_{\text{bdry}}$,
let us define the loss function as
\begin{equation}
\label{eq:loss_dy}
\mathcal{L}[w(\cdot;\omega)]:= \frac{1}{N_{\text{tot}}}\bigg( \sum_{i=1}^{N_p} \underbrace{\mathcal{L}_0 [w(x_p^i,y_p^i,0;\omega)]}_{\text{initial loss}} +\sum_{j=1}^{N_{\text{int}}}   \underbrace{\mathcal{L}_{p}[w(x_{\text{int}}^j,y_{\text{int}}^j,t_j;\omega)]}_{\text{ODE system loss}}+\sum_{k=1}^{N_{\text{bdry}}} \underbrace{\mathcal{L}_{\text{bdry}}[w(x_{\text{bdry}}^k,y_{\text{bdry}}^k,t_k;\omega)]}_{\text{boundary loss}}\bigg)
\end{equation}
for 
\[
\mathcal{L}_0[w(x,y,0;\omega)] := (w(x,y,0)[1]-x)^2 + (w(x,y,0;\omega)[2]-y)^2,
\]
and
\[
\mathcal{L}_{p}[w(\cdot;\omega)]:=  \bigg(\frac{\partial w(\cdot;\omega)[1]}{ \partial t} + \beta w(\cdot;\omega)[2] w(\cdot;\omega)[1] \bigg)^2 + \bigg(\frac{\partial w(\cdot;\omega)[2]}{\partial t} - \beta w(\cdot;\omega)[1] w(\cdot)[2] +\gamma w(\cdot;\omega)[2]\bigg)^2,
\]
where $w(\cdot)[1]$ and $w(\cdot)[2]$ denote the first and second component of $w$. 

Lastly, $\mathcal{L}_{\text{bdry}}$ is defined via
\[
\mathcal{L}_{\text{bdry}}[w(x,y,kdt;\omega)]:=  \|w(x,y,kdt;\omega)-g(x,y,kdt)\|_2^2,
\]
where the reference vector $g(x,y,kdt)$ is computed as follows:
\[
g(x,y,kdt)= (S_k, I_k),
\]
where the reference vector $g(x,y,kdt)$ is obtained by numerically solving the uncontrolled dynamics~\eqref{eq:ode_uncontrolled} using the fourth-order Runge-Kutta method, with initial conditions $(S_0, I_0) = (x,y)$ and a stepsize $dt>0$. The training details are provided in Algorithm~\ref{alg:uncontolled_dynamics}.

\subsubsection{Learning the switching time $\tau$}
Given $u$, $u^{r_0}$, and a neural network $w(x,y,t)$ that approximates the uncontrolled SIR dynamics satisfying~\eqref{eq:ode_uncontrolled} with $(S(0),I(0))=(x,y)$, we minimize the following loss function:
\begin{equation}
\label{eq:tau_dp_loss}
\mathcal{L}[\tau(\cdot;\theta)]:=\frac{1}{N}\sum_{i=1}^N \bigg\{\bigg(u(x_i,y_i) -  \tau(x_i,y_i;\theta) - u^{r_0}(S(\tau(x_i,y_i;\theta)),I(\tau(x_i,y_i;\theta)))\bigg)^2+\text{pen}(x_i,y_i,\tau(\cdot;\theta))\bigg\},
\end{equation}
over $\theta$ for randomly sampled points $\{(x^i,y^i)\}_{i=1}^N \in D$ where

\begin{equation*}
\text{pen}(x,y,\tau)=
\begin{cases}
\bigg(\partial_x u^{r_0}(w(x,y,\tau)[1], w(x,y,\tau)[2])\bigg)^2 \quad&\text{if}\quad (x,y) \in \mathcal{S}^C, \\
0 \quad &\text{otherwise}.
\end{cases}
\end{equation*}
Here, $w(x,y,\tau)$ approximates the solution $(S(\tau),I(\tau))$ to uncontrolled SIR with $(S(0),I(0))=(x,y)$ and is learned from the previous section. The penalty term ensures that $\tau$ satisfies~\eqref{eq:u0_prop}.

\section{Experimental results}
\subsection{Parameters and architecture}
In all experiments presented in this paper, the parameters of the controlled SIR model are set as follows: \( \beta = 2 \), \( \gamma = 10 \), and \( \mu = 0.01 \). We optimize using the Adam optimizer~\cite{kingma2014adam} with a fixed learning rate of 0.0001.

\subsubsection{Learning \( u \), \( u^{r_0} \)}
\begin{figure}[ht]
    \centering
    \includegraphics[width=0.6\linewidth]{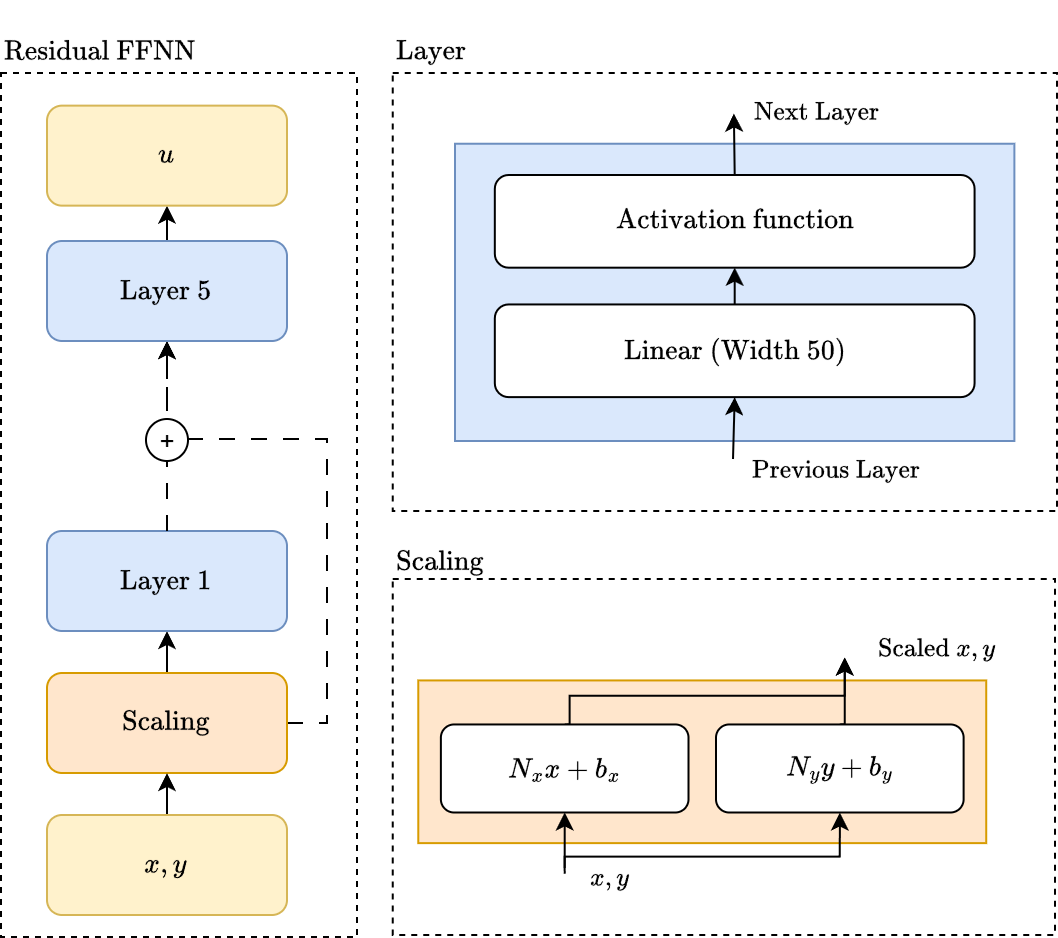}
    \caption{Feed-forward network with residual connections and normalization scheme.}
    \label{fig:ffnn}
\end{figure}
We employ feed-forward neural networks with residual connections, where the width is fixed as 50 and the depth consists of 5 hidden layers unless otherwise stated, as illustrated in Figure~\ref{fig:ffnn}. Regarding activation functions, we use the hyperbolic tangent function.

To test the efficiency of our method for approximating $u$ and $u^{r_0}$, we consider the domain $[0,20]\times[0.01,1]$ by setting $\ell_x=20$ and $\ell_y + \mu=1$, while varying the scaling factors $N_x$ and $N_y$ in Section~\ref{sec:training_u}. During training, we sample 1000 interior collocation points at each iteration. 

To define the boundary loss~\eqref{eq:bdry_loss}, we randomly select 100 data points from each boundary $D_i$ defined in~\eqref{def:bdry} for $i=1,2,3$, resulting in a total of 300 points, as illustrated in Figure~\ref{fig:collocations}. For the reference values obtained by Algorithm~\ref{alg:approx_era}, we set $dt=0.001$ and $d\tau=0.001$.

\subsubsection{Learning the uncontrolled dynamics}
To learn the uncontrolled dynamics, we use the same neural network architecture and activation functions as described above but consider a slightly larger spatio-temporal sampling domain given by $[0,20] \times [0.01,2] \times [0,2.5]$, setting $\ell_x=20$, $\ell_y+\mu=2$, and $T=2.5$ in Section~\ref{sub:learning_dy}. The value of $T$ is determined heuristically, based on the observation that the susceptible and infectious populations approach zero after a sufficiently long duration, which we set to $T=2.5$.

To define the boundary loss, we set $N_T=250$, and hence, $dt=0.01$ in Algorithm~\ref{alg:uncontolled_dynamics} and sample 4000 points from $\partial D \times \mathcal{T}$ by setting $N_{\text{bdry}}=4000$ in~\eqref{eq:loss_dy}. 

For the initial loss, we sample 1000 points per batch from $D$ by setting $N_p=1000$, shown as green dots in the figure. Additionally, to define the ODE system loss in~\eqref{eq:loss_dy}, 1000 collocation points \((S(0), I(0), t)\) are sampled per batch uniformly by setting $N_{\text{int}}=1000$.

\subsubsection{Learning the optimal switching time \(\tau \)}
We use a neural network without residual connections, with a width of 200 and 5 hidden layers, where leaky ReLU functions are applied in all hidden layers, and the Softplus function is employed in the final layer.

The domain of interest is given by $[0,20] \times [0.01,1]$. Recalling the loss function~\eqref{eq:tau_dp_loss}, we note that the values of $u$ are required in the same domain, while a larger domain must be considered for $u^{r_0}$ and $w$ since $I(\tau(x,y;\theta))$ may exceed 1. Therefore, we approximate $u^{r_0}$ in the domain $[0,20]\times[0.01,10]$, using variable scaling factors of $N_x=1$ and $N_y=4$.

Similar to the previous experiments, we randomly sample 1,000 collocation points per batch uniformly, as shown in Figure~\ref{fig:collocations} (right), and train using the DPP loss function~\eqref{eq:tau_dp_loss}.

\begin{figure}[ht]
    \centering
    \includegraphics[width=0.9\linewidth]{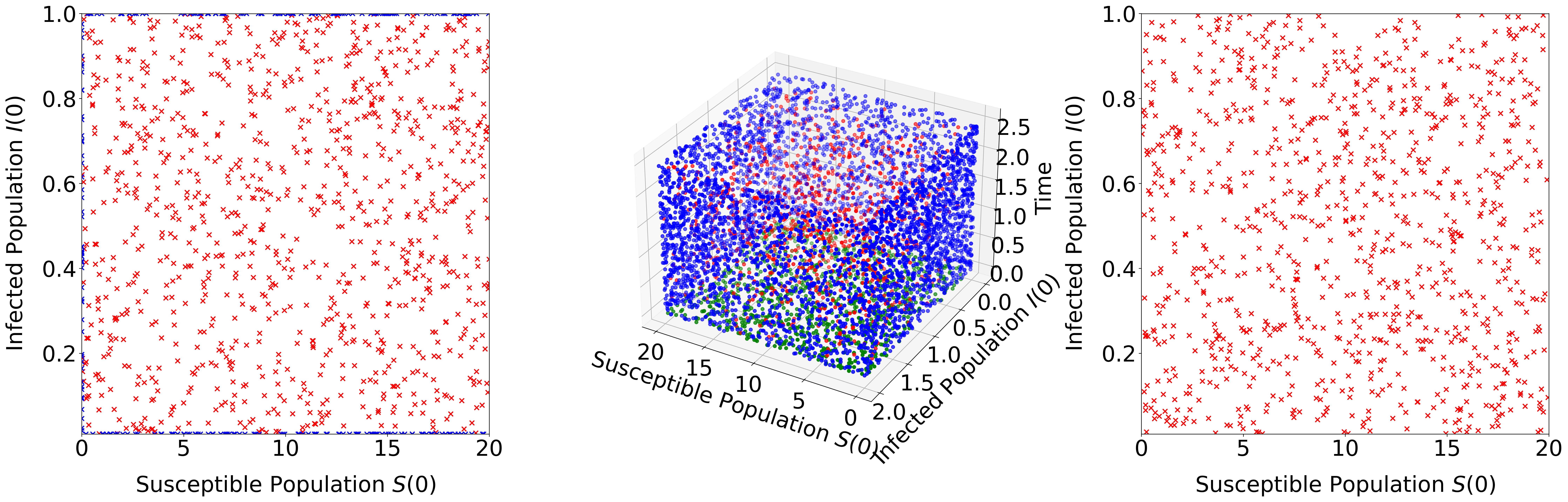}
    \caption{Collocation points and data points used in training \(u\), \(u^{r_0}\), \(w\) and \(\tau\). For training \(u\) and \(u^{r_0}\), collocation points (red dots) and data points (blue dots) are used (left). For training \(w\), collocation points (red dots), data points (blue dots), and initial condition (green area) are used as shown in the middle. For training \(\tau\), only collocation points are used (right).}
    \label{fig:collocations}
\end{figure}

\subsection{Result of learning $u$ and $u^{r_0}$ via PINN}
\begin{table}[ht]
\begin{center}
\caption{Evaluation of mean square error (MSE) of variable scaling}
\begin{tabular}{c|c|c|c} 
    \hline
    $N_x$ & $N_y$ & MSE of \(u\) & MSE of \(u^{r_0}\) \\
    \hline
     1. & 1. & \(2.604\times10^{-3}\) & \(4.666\times10^{-4}\) \\
     0.05 & 1. & \(2.876\times10^{-3}\)  & \(4.176\times10^{-4}\) \\
     1. & 10. & \(3.264\times10^{-4}\) & \(2.670\times10^{-5}\)\\
     2. & 20. & \(2.901\times10^{-4}\)  & \(2.424\times10^{-5}\)\\
    \hline
\end{tabular}
\label{tb:eval_vs}
\end{center}
\end{table}

\begin{table}[ht]
\begin{center}
\caption{Evaluation of mean square error (MSE) of different depths in the domain $[0,20] \times [0.01,10]$ with $N_x=2$ and $N_y=20$.}
\begin{tabular}{c|c|c} \hline
    Number of hidden layers & MSE of \(u\) & MSE of \(u^{r_0} \) \\

    \hline
    5 & \(2.291\times10^{-4}\) & \(1.982\times10^{-4}\)  \\
    1 & \(5.877\times10^{-3}\) & \(4.650\times10^{-4}\)  \\
    \hline

\end{tabular}
\label{tb:eval_u_depth}
\end{center}
\end{table}
We use the minimum eradication time \(u\) obtained from Algorithm~\ref{alg:approx_era} as a reference and compare it with the approximate minimum eradication time computed by our PINN algorithm. Similarly, we compute the eradication time under the control \(u^{r_0}\) using the Runge-Kutta method with a timestep of \(dt=0.001\). The mean square error (MSE) is then evaluated over the entire set of reference points obtained from Algorithm~\ref{alg:approx_era} across the domain. 

Table~\ref{tb:eval_vs} summarizes the MSE values corresponding to different choices of the scaling factors \(N_x\) and \(N_y\). As observed, setting \(N_x=2\) and \(N_y=20\) results in the lowest MSE, which aligns with the error analysis presented in Section~\ref{sec:ntk}. 

Additionally, we investigate the effect of network depth in the domain \( [0,20] \times [0.01,10] \). Table~\ref{tb:eval_u_depth} presents the results for different numbers of hidden layers. Training with a model containing 5 hidden layers yields a significantly lower MSE compared to training with a single hidden layer, demonstrating that deeper networks improve the approximation accuracy. The level set for values of $u$ and $u^{r_0}$ is presented in Figure~\ref{fig:layers}.

\begin{figure}[ht]
    \centering
    \begin{subfigure}[b]{1.0\textwidth}
        \centering
        \includegraphics[width=\textwidth, trim=0 0 0 0.05\textheight, clip]{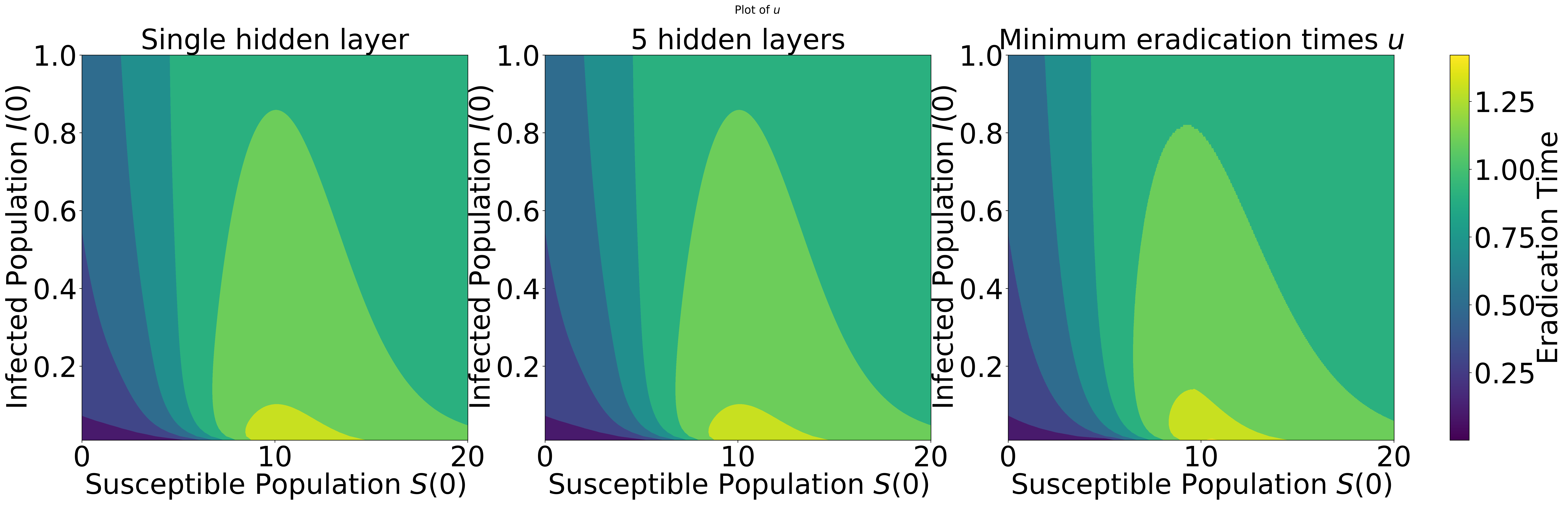}

        \caption{Values of \(u\) learned via PINN with a different number of hidden layers and \(N_x=2, N_y=20\): a single hidden layer (left), 5 hidden layers (middle), and the reference values of \(u\) computed by Algorithm~\ref{alg:approx_era} (right).}
        \label{fig:u_layers}
    \end{subfigure}
    \begin{subfigure}[b]{1.0\textwidth}
        \centering
        \includegraphics[width=\textwidth, trim=0 0 0 0.05\textheight, clip]{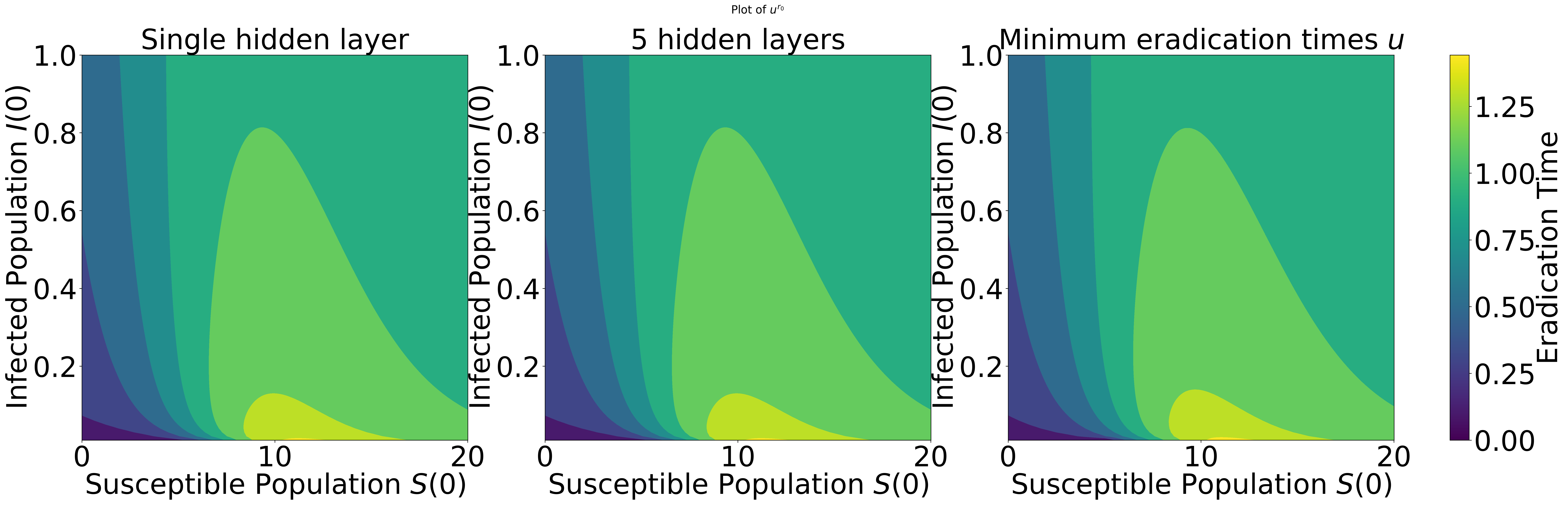}
        \caption{Values of trained \(u^{r_0}\) learned via PINN with a different number of hidden layers and \(N_x=2, N_y=20\): a single hidden layer (left), 5 hidden layers (middle), and reference values of \(u^{r_0}\) computed by Algorithm~\ref{alg:approx_era} (right).}
        \label{fig:u_r0_layers}
    \end{subfigure}
    \caption{\(u\) and \(u^{r_0}\) trained with different hidden layers size}
    \label{fig:layers}
\end{figure}

\subsection{Result of learning uncontrolled SIR dynamics}

To evaluate the trained uncontrolled SIR dynamics \(w\), we use the Runge-Kutta method. For reference data, we first sample \(\{(x_i,y_i)\}_{i=1}^{1000} \subset [0,20] \times [0.01,1]\) and approximate \((S(t_k), I(t_k))\), which satisfy the uncontrolled dynamics~\eqref{eq:ode_uncontrolled} with initial conditions \((S(0), I(0)) = (x_i, y_i)\) for all \(i=1,\dots,1000\) and \(t_k \in \{0.025 k \mid 0 \leq k \leq 100 \}\), using the Runge-Kutta method.

The mean square error (MSE) of \(w\) in estimating the population dynamics is \(7.557\times10^{-3}\), indicating that the neural network effectively estimates the uncontrolled dynamics. Figure~\ref{fig:si_pred_compare} illustrates the comparison between trajectories approximated by the trained neural network and those obtained via the Runge-Kutta method.

\begin{figure}[ht]
    \centering
    \includegraphics[width=0.9\textwidth]{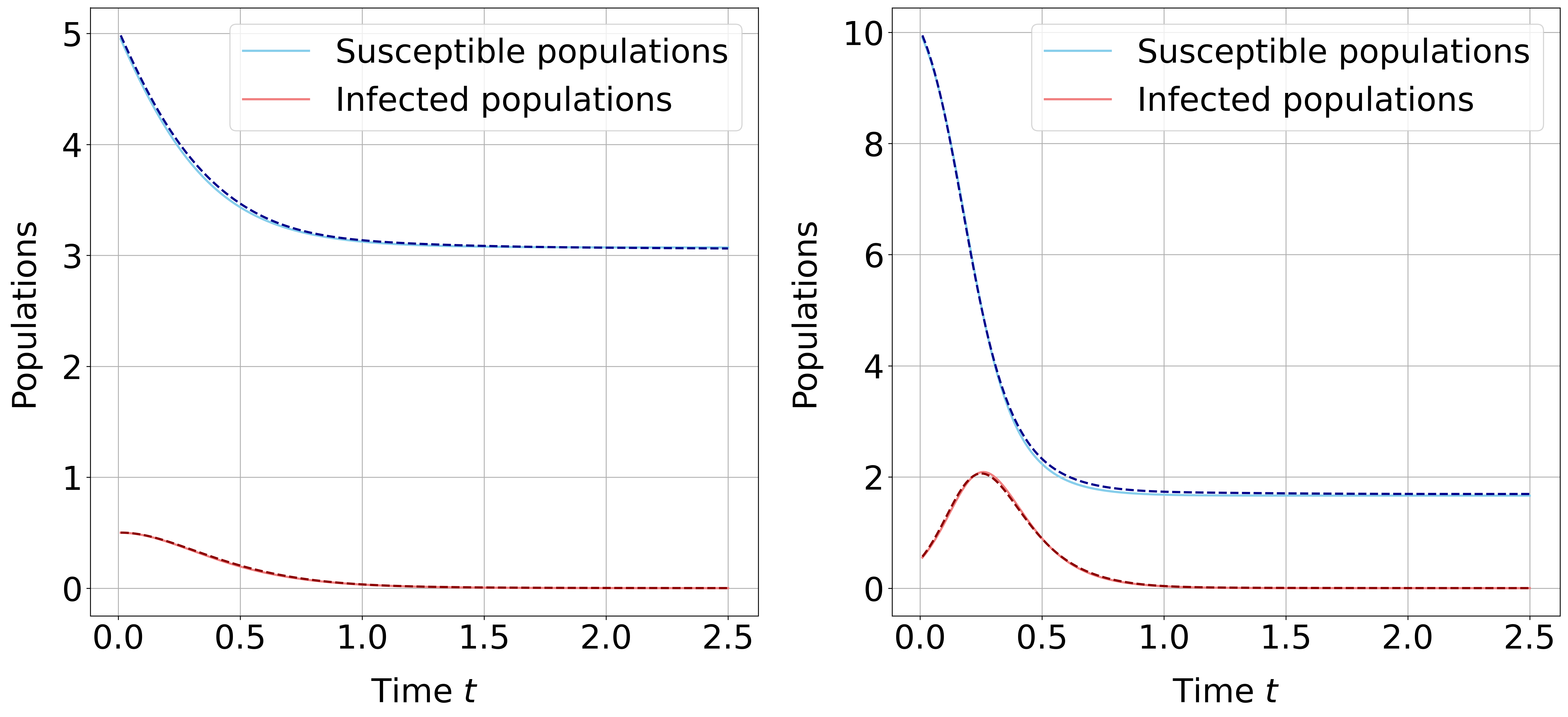}
    \caption{Comparison between \(w(S(0), I(0), t; \theta)\) and the ground truth computed via the Runge-Kutta method for selected initial conditions \((S(0), I(0))\) that are not part of the training set or boundary conditions. The left  shows the trajectories for \(S(0) = 5.0, I(0) = 0.5\), while the right panel corresponds to \(S(0) = 10.0, I(0) = 0.5\). The dashed lines represent the results obtained using the Runge-Kutta method, whereas the solid lines indicate the estimations from the trained neural network.
 }
    \label{fig:si_pred_compare}
\end{figure}

\subsection{Result of learning optimal switching time}
Recalling~\eqref{eq:u_prop} and~\eqref{eq:u0_prop}, the region where the optimal switching time is nonzero is characterized by the inequality \(\partial_x u(x,y) \leq 0\), as demonstrated in Figure~\ref{fig:tau_s}.

\begin{figure}[H]
    \centering
    \includegraphics[width=\linewidth]{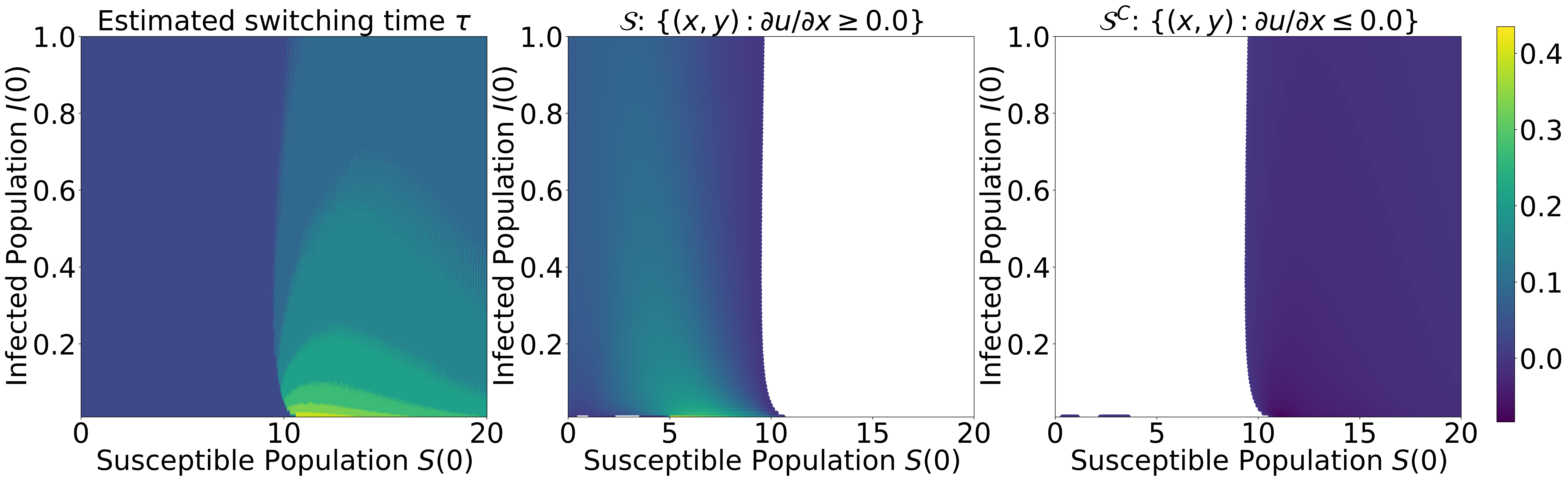}
\caption{Heatmap of regions \( S \) and its complement.
    The figure illustrates the ground truth of the optimal switching time \( \tau \) (left), the region \( S = \{(x,y) : \partial_x u(x,y) \geq 0\} \) (middle), and the complementary region \( S^C = \{(x,y) : \partial_x u(x,y) \leq 0\} \) (right).}
    \label{fig:tau_s}
\end{figure}

To evaluate the accuracy of the estimated switching time, we compute the mean square error (MSE) of the estimated switching time \(\tau\) using samples \(\{(x_i,y_i)\}_{i=1}^{1000} \subset \mathcal{S}^C\).

The optimal switching time for initial conditions \((S(0),I(0)) : \{(0.01i,0.01j) \mid 0 \leq i \leq 2000, 0 \leq j \leq 100\}\) is approximated using reference values obtained from Algorithm~\ref{alg:approx_era}, with a time discretization of $dt=0.001$, and $d\tau=0.001$.

Furthermore, we assess the performance of our model using different numbers of hidden layers. When using a single hidden layer, the model achieves an MSE of \(6.662\times10^{-4}\). However, increasing the depth to five hidden layers results in a slightly higher MSE of \(6.667\times10^{-4}\), indicating that additional depth does not necessarily improve accuracy in this setting.

Figure~\ref{fig:tau_compare} confirms that our trained model achieves accurate estimation of \(\tau\) with minimal error.

\begin{figure}[ht]
    \centering
    \includegraphics[width=\linewidth]{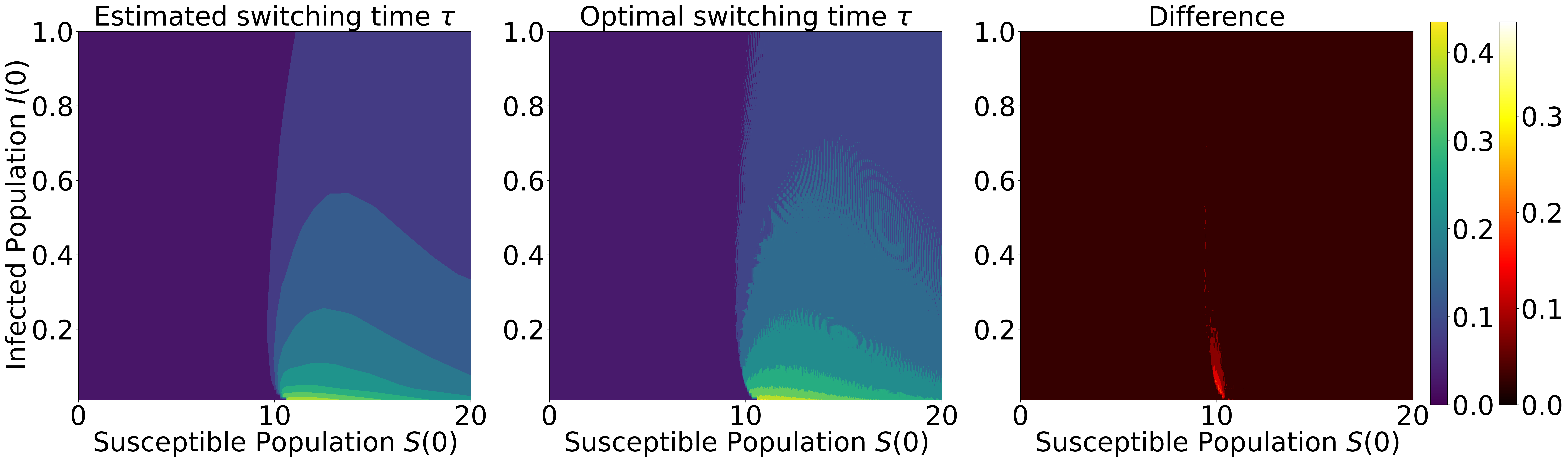}
\caption{Results of trained switching time \( \tau \). The figure presents the estimated switching time \( \tau(\cdot; \theta) \) (left), the ground truth \( \tau \) (middle), and the difference between the ground truth and the estimated values, \( \tau - \tau(\cdot; \theta) \) (right).}
    \label{fig:tau_compare}
\end{figure}

\section{Conclusion}
We proposed a novel approach for approximating the minimum eradication time and synthesizing optimal vaccination strategies for a controlled SIR model via variable-scaling PINNs and the dynamic programming principle. A salient feature of our method is that it does not require domain discretization and offers a computationally efficient solution to HJB equations related to the minimum eradication time. The experimental results confirmed the accuracy and robustness of our method in approximating the eradication time and deriving optimal vaccination strategies. Furthermore, the theoretical justification for our approach is established based on NTK theory. 

Despite these contributions, our study has several limitations. The current framework assumes a time-homogeneous controlled SIR model with constant infection and recovery rates, which might limit its applicability to real-world scenarios with time-inhomogeneous dynamics or heterogeneous populations. Additionally, while our method improves computational efficiency, the training of PINNs remains sensitive to hyperparameter choices and boundary conditions, which require further investigation. Future work will focus on extending this framework to time-inhomogeneous SIR models, for which theoretical support has been proposed in~\cite{jang2023minimum}. Additionally, comparative studies with other data-driven techniques, such as reinforcement learning, will provide a deeper understanding of the advantages and limitations of PINN-based approaches in epidemiological modeling and control.

\section*{Acknowledgments}
This research was supported by Seoul National University of Science and Technology. 

\section*{Conflict of interest}

\bibliographystyle{IEEEtran}

\bibliography{ref}

\begin{thebibliography}{10}
\providecommand{\url}[1]{#1}
\csname url@samestyle\endcsname
\providecommand{\newblock}{\relax}
\providecommand{\bibinfo}[2]{#2}
\providecommand{\BIBentrySTDinterwordspacing}{\spaceskip=0pt\relax}
\providecommand{\BIBentryALTinterwordstretchfactor}{4}
\providecommand{\BIBentryALTinterwordspacing}{\spaceskip=\fontdimen2\font plus
\BIBentryALTinterwordstretchfactor\fontdimen3\font minus \fontdimen4\font\relax}
\providecommand{\BIBforeignlanguage}[2]{{%
\expandafter\ifx\csname l@#1\endcsname\relax
\typeout{** WARNING: IEEEtran.bst: No hyphenation pattern has been}%
\typeout{** loaded for the language `#1'. Using the pattern for}%
\typeout{** the default language instead.}%
\else
\language=\csname l@#1\endcsname
\fi
#2}}
\providecommand{\BIBdecl}{\relax}
\BIBdecl

\bibitem{kermack1927SIR}
W.~O. Kermack and A.~G. McKendrick, ``A contribution to the mathematical theory of epidemics,'' \emph{Proc. R. Soc. Lond., Ser. A, Math. Phys. Eng. Sci.}, vol. 115, no. 772, pp. 700--721, 1927.

\bibitem{barro2018optimal}
M.~Barro, A.~Guiro, and D.~Quedraogo, ``Optimal control of a {SIR} epidemic model with general incidence function and a time delays,'' \emph{CUBO A Math. J.}, vol.~20, no.~2, pp. 53--66, 2018.

\bibitem{pierre-alexandre2007optimal}
P.~A. Bliman, M.~Duprez, Y.~Privat, and N.~Vauchelet, ``Optimal immunity control and final size minimization by social distancing for the {SIR} epidemic model,'' \emph{J. Optim. Theory Appl.}, vol. 189, no.~2, pp. 408--436, 2021.

\bibitem{bolzoni2017time}
L.~Bolzoni, E.~Bonacini, C.~Soresina, and M.~Groppi, ``Time-optimal control strategies in {SIR} epidemic models,'' \emph{Math. Biosci.}, vol. 292, pp. 86--96, 2017.

\bibitem{ev2016optimalcontrol}
E.~V. Grigorieva, E.~N. Khailov, and A.~Korobeinikov, ``Optimal control for a {SIR} epidemic model with nonlinear incidence rate,'' \emph{Math. Model. Nat. Phenom.}, vol.~11, no.~4, pp. 89--104, 2016.

\bibitem{hynd2021eradication}
R.~Hynd, D.~Ikpe, and T.~Pendleton, ``An eradication time problem for the {SIR} model,'' \emph{J. Differ. Equ.}, vol. 303, pp. 214--252, 2021.

\bibitem{hynd2022critical}
------, ``Two critical times for the {SIR} model,'' \emph{J. Math. Anal. Appl.}, vol. 505, 2022.

\bibitem{pontryagin2018mathematical}
L.~S. Pontryagin, \emph{Mathematical theory of optimal processes}.\hskip 1em plus 0.5em minus 0.4em\relax Routledge, 2018.

\bibitem{jang2023minimum}
J.~Jang and Y.~Kim, ``On a minimum eradication time for the {SIR} model with time-dependent coefficients,'' \emph{arXiv preprint arXiv:2311.14657}, 2023.

\bibitem{tran2021hamilton}
H.~V. Tran, \emph{Hamilton--{J}acobi {E}quations: {T}heory and {A}pplications}.\hskip 1em plus 0.5em minus 0.4em\relax American Mathematical Society, Graduate Studies in Mathematics, 2021, vol. 213.

\bibitem{VS-PINN}
S.~Ko and S.~H. Park, ``{VS-PINN}: A fast and efficient training of physics-informed neural networks using variable-scaling methods for solving pdes with stiff behavior,'' \emph{arXiv preprint arXiv:2406.06287}, 2024.

\bibitem{jacot2018neural}
A.~Jacot, F.~Gabriel, and C.~Hongler, ``Neural tangent kernel: Convergence and generalization in neural networks,'' \emph{Adv. Neural. Inf. Process. Syst.}, vol.~31, 2018.

\bibitem{raissi2019physics}
M.~Raissi, P.~Perdikaris, and G.~E. Karniadakis, ``Physics-informed neural networks: A deep learning framework for solving forward and inverse problems involving nonlinear partial differential equations,'' \emph{J. Comput. Phys.}, vol. 378, pp. 686--707, 2019.

\bibitem{kharazmi2021identifiability}
E.~Kharazmi, M.~Cai, X.~Zheng, Z.~Zhang, G.~Lin, and G.~E. Karniadakis, ``Identifiability and predictability of integer-and fractional-order epidemiological models using physics-informed neural networks,'' \emph{Nat. Comput. Sci.}, vol.~1, no.~11, pp. 744--753, 2021.

\bibitem{yazdani2020systems}
A.~Yazdani, L.~Lu, M.~Raissi, and G.~E. Karniadakis, ``Systems biology informed deep learning for inferring parameters and hidden dynamics,'' \emph{PLoS Comput. Biol.}, vol.~16, no.~11, p. e1007575, 2020.

\bibitem{cai2021physics}
S.~Cai, Z.~Mao, Z.~Wang, M.~Yin, and G.~E. Karniadakis, ``Physics-informed neural networks (pinns) for fluid mechanics: A review,'' \emph{Acta Mech. Sin.}, vol.~37, no.~12, pp. 1727--1738, 2021.

\bibitem{raissi2020hidden}
M.~Raissi, A.~Yazdani, and G.~E. Karniadakis, ``Hidden fluid mechanics: Learning velocity and pressure fields from flow visualizations,'' \emph{Science}, vol. 367, no. 6481, pp. 1026--1030, 2020.

\bibitem{jin2021nsfnets}
X.~Jin, S.~Cai, H.~Li, and G.~E. Karniadakis, ``{NSF}nets ({N}avier-{S}tokes flow nets): {P}hysics-informed neural networks for the incompressible {N}avier-{S}tokes equations,'' \emph{J. Comput. Phys.}, vol. 426, p. 109951, 2021.

\bibitem{wang2023deep}
X.~Wang, J.~Li, and J.~Li, ``A deep learning based numerical pde method for option pricing,'' \emph{Comput. Econ.}, vol.~62, no.~1, pp. 149--164, 2023.

\bibitem{bai2022application}
Y.~Bai, T.~Chaolu, and S.~Bilige, ``The application of improved physics-informed neural network (ipinn) method in finance,'' \emph{Nonlinear Dyn.}, vol. 107, no.~4, pp. 3655--3667, 2022.

\bibitem{kissas2020machine}
G.~Kissas, Y.~Yang, E.~Hwuang, W.~R. Witschey, J.~A. Detre, and P.~Perdikaris, ``Machine learning in cardiovascular flows modeling: Predicting arterial blood pressure from non-invasive 4d flow mri data using physics-informed neural networks,'' \emph{Comput. Methods Appl. Mech. Eng.}, vol. 358, p. 112623, 2020.

\bibitem{sahli2020physics}
F.~Sahli~C., Y.~Yang, P.~Perdikaris, D.~E. Hurtado, and E.~Kuhl, ``Physics-informed neural networks for cardiac activation mapping,'' \emph{Front. Phys.}, vol.~8, p.~42, 2020.

\bibitem{wang2021understanding}
S.~Wang, Y.~Teng, and P.~Perdikaris, ``Understanding and mitigating gradient flow pathologies in physics-informed neural networks,'' \emph{SISC}, vol.~43, no.~5, pp. A3055--A3081, 2021.

\bibitem{liu2022physics}
Y.~Liu, L.~Cai, Y.~Chen, and B.~Wang, ``Physics-informed neural networks based on adaptive weighted loss functions for {H}amilton-{J}acobi equations,'' \emph{Math. Biosci. Eng.}, vol.~19, no.~12, 2022.

\bibitem{wang2022and}
S.~Wang, X.~Yu, and P.~Perdikaris, ``When and why {PINN}s fail to train: A neural tangent kernel perspective,'' \emph{J. Comput. Phys.}, vol. 449, p. 110768, 2022.

\bibitem{ji2021stiff}
W.~Ji, W.~Qiu, Z.~Shi, S.~Pan, and S.~Deng, ``Stiff-pinn: Physics-informed neural network for stiff chemical kinetics,'' \emph{J. Phys. Chem. A}, vol. 125, no.~36, pp. 8098--8106, 2021.

\bibitem{mcclenny2023self}
L.~D. McClenny and U.~M. Braga-Neto, ``Self-adaptive physics-informed neural networks,'' \emph{J. Comput. Phys.}, vol. 474, p. 111722, 2023.

\bibitem{mowlavi2023optimal}
S.~Mowlavi and S.~Nabi, ``Optimal control of pdes using physics-informed neural networks,'' \emph{J. Comput. Phys.}, vol. 473, p. 111731, 2023.

\bibitem{meng2024physics}
Y.~Meng, R.~Zhou, A.~Mukherjee, M.~Fitzsimmons, C.~Song, and J.~Liu, ``Physics-informed neural network policy iteration: Algorithms, convergence, and verification,'' \emph{arXiv preprint arXiv:2402.10119}, 2024.

\bibitem{lee2024hamilton}
J.~Y. Lee and Y.~Kim, ``Hamilton--{J}acobi based policy-iteration via deep operator learning,'' \emph{arXiv preprint arXiv:2406.10920}, 2024.

\bibitem{tang2023policy}
W.~Tang, H.~V. Tran, and Y.~P. Zhang, ``Policy iteration for the deterministic control problems--a viscosity approach,'' \emph{arXiv preprint arXiv:2301.00419}, 2023.

\bibitem{yin2023optimal}
S.~Yin, J.~Wu, and P.~Song, ``Optimal control by deep learning techniques and its applications on epidemic models,'' \emph{J. Math. Biol.}, vol.~86, no.~3, p.~36, 2023.

\bibitem{hogg2013introduction}
R.~V. Hogg, J.~W. McKean, A.~T. Craig \emph{et~al.}, \emph{Introduction to {M}athematical {S}tatistics}.\hskip 1em plus 0.5em minus 0.4em\relax Pearson Education India, 2013.

\bibitem{kingma2014adam}
D.~Kingma, ``{A}dam: {A} method for stochastic optimization,'' \emph{arXiv preprint arXiv:1412.6980}, 2014.

\end{thebibliography}

\end{document}